\documentclass[11pt,a4paper]{article}
\usepackage{amsmath,amsfonts,amssymb,amsthm,latexsym}
\usepackage{graphics,epsfig,color,enumerate}
\usepackage{setspace,mathtools,enumitem,dsfont,verbatim}
\usepackage{caption}
\usepackage{subcaption}
\usepackage[english]{babel}
\usepackage[T1]{fontenc}
\usepackage{etoolbox}

\makeatletter
\def\@captype{figure}
\patchcmd{\maketitle}{\@fnsymbol}{\@alph}{}{}  % Footnote numbers from symbols to small letter
\makeatother
\newtheorem{theorem}{Theorem}[section]
\newtheorem{lemma}[theorem]{Lemma}
\newtheorem{proposition}[theorem]{Proposition}
\newtheorem{corollary}[theorem]{Corollary}
\newtheorem{remark}[theorem]{Remark}

\newtheorem{definition}[theorem]{Definition}

\numberwithin{equation}{section}

\newcommand{\R}{\mathbb{R}}

\newcommand{\N}{\mathbb{N}}

\newcommand{\PP}{\mathbb{P}}

\newcommand{\cG}{\mathcal{G}}

\newcommand{\vC}{\vec{C}}
\newcommand{\vt}{\vec{\theta}}
\newcommand{\Pmic}{\mathrm{P}_{\mathrm{mic}}}
\newcommand{\Pcan}{\mathrm{P}_{\mathrm{can}}}
\newcommand{\dd}{\mathrm{d}}
\newcommand{\eee}{\mathrm{e}}

\newcommand{\be}{\begin{equation}}
\newcommand{\ee}{\end{equation}}

\newcommand{\limn}{\lim_{n\to\infty}}

\begin{document}

\title{Ensemble equivalence for dense graphs}

\author{
   F.~den Hollander\thanks{Mathematical Institute, Leiden University, P.O. Box 9512, 2300, RA Leiden, The Netherlands}
   \and
  M.~Mandjes\thanks{Korteweg-de Vries Institute, University of Amsterdam, P.O. Box 94248, 1090 GE Amsterdam, The Netherlands}
\and
A.~Roccaverde$^{\rm{a}}$
\and
  N.~J.~Starrreveld$^{\rm{b}}$
} 

\maketitle

%%%%%%%%%%%%%%%%%%%%%%%%%%%%%%%%%%%

\begin{abstract}
In this paper we consider a random graph on which topological restrictions are imposed, 
such as constraints on the total number of edges, wedges, and triangles. We work in 
the dense regime, in which the number of edges per vertex scales proportionally to the 
number of vertices $n$. Our goal is to compare the micro-canonical ensemble (in which 
the constraints are satisfied for every realisation of the graph) with the canonical ensemble 
(in which the constraints are satisfied on average), both subject to maximal entropy. We 
compute the relative entropy of the two ensembles in the limit as $n$ grows large, where 
two ensembles are said to be \emph{equivalent} in the dense regime if this relative entropy 
divided by $n^2$ tends to zero. Our main result, whose proof relies on large deviation theory 
for graphons, is that breaking of ensemble equivalence occurs when the constraints are \emph{frustrated}. Examples are provided for three different choices of 
constraints.

\end{abstract}

\section{Introduction}
\label{S1}

Section~\ref{S1.1} gives background and motivation, Section~\ref{S1.1alt} describes 
relevant literature, while Section~\ref{S1.2} outlines the remainder of the paper. 
%%%

\subsection{Background and motivation}
\label{S1.1}

For large networks a detailed description of the architecture of the network is infeasible 
and must be replaced by a \emph{probabilistic} description, where the network is assumed 
to be a random sample drawn from a set of allowed graphs that are consistent with a set 
of empirically observed features of the network, referred to as \emph{constraints}. 
Statistical physics deals with the definition of the appropriate probability distribution 
over the set of graphs and with the calculation of its relevant properties (Gibbs~\cite{G02}). 
The two main choices\footnote{The microcanonical ensemble and the canonical ensemble 
work with a fixed number of vertices. There is a third ensemble, the \emph{grandcanonical 
ensemble}, where also the size of the graph is considered as a soft constraint.} of probability 
distribution are: 
\begin{itemize}
\item[(1)] 
The \emph{microcanonical ensemble}, where the constraints are \emph{hard} (i.e., are 
satisfied by each individual graph).
\item[(2)] 
The \emph{canonical ensemble}, where the constraints are \emph{soft} (i.e., hold as 
ensemble averages, while individual graphs may violate the constraints).
\end{itemize}

For networks that are large but finite, the two ensembles are obviously different and, 
in fact, represent different empirical situations: they serve as \emph{null-models} for 
the network after incorporating what is known about the network \emph{a priori} via 
the constraints. Each ensemble represents the unique probability distribution with 
\emph{maximal entropy} respecting the constraints. In the limit as the size of the graph 
diverges, the two ensembles are traditionally \emph{assumed} to become equivalent 
as a result of the expected vanishing of the fluctuations of the soft constraints, i.e., the 
soft constraints are expected to become asymptotically hard. This assumption of 
\emph{ensemble equivalence}, which is one of the corner stones of statistical physics, 
does however \emph{not} hold in general (we refer to Touchette~\cite{T14} for more 
background). 

In Squartini \emph{et al.}~\cite{SdMdHG15} the question of the possible breaking of ensemble 
equivalence was investigated for two types of constraint:
\begin{itemize}
\item[(I)] 
The total number of edges. 
\item[(II)] 
The degree sequence.
\end{itemize} 
In the \emph{sparse regime}, where the empirical degree distribution converges to a limit 
as the number of vertices $n$ tends to infinity such that the maximal degree is $o(\sqrt{n})$, 
it was shown that the relative entropy of the micro-canonical ensemble w.r.t.\ the canonical 
ensemble divided by $n$ (which can be interpreted as the relative entropy per vertex) tends 
to $s_\infty$, with $s_\infty=0$ in case the constraint concerns the total number of edges, and 
$s_\infty>0$ in case the constraint concerns the degree sequence. For the latter case, an 
explicit formula was derived for $s_\infty$, which allows for a \emph{quantitative analysis} of 
the breaking of ensemble equivalence. 

In the present paper we analyse what happens in the \emph{dense regime}, where the number 
of edges per vertex is of order $n$. We consider case (I), yet allow for constraints not only 
on the total number of edges but also on the total number of wedges, triangles, etc. We show
that the relative entropy divided by $n^2$ (which, up to a constant, can be interpreted as the 
relative entropy per edge) tends to $s_\infty$, with $s_\infty>0$ when the constraints are \emph{frustrated}. Our analysis is based on a large deviation principle for 
graphons. 

%%%

\subsection{Relevant literature}
\label{S1.1alt}

In the past few years, several papers have studied the microcanonical ensemble and the 
canonical ensemble. Most papers focus on dense graphs, but there are some interesting 
advances for sparse graphs as well. Closely related to the canonical ensemble are the 
\emph{exponential random graph model} (Bhamidi {\it et al.}~\cite{BBS11}, Chatterjee 
and Diaconis~\cite{CD13}) and the \emph{constrained exponential random model}  (Aristoff 
and Zhu~\cite{AZ11}, Kenyon and Yin~\cite{KY17}, Yin~\cite{Y15}, Zhu~\cite{Z17}). 

In Aristoff and Zhu~\cite{AZ11}, Kenyon {\it et al.}~\cite{KRRS17}, Radin and Sadun~\cite{RS15}, 
the authors study the microcanonical ensemble, focusing on the constrained entropy density. 
In \cite{AZ11} directed graphs are considered with a \emph{hard} constraint on the number 
of directed edges and $j$-stars, while in \cite{KRRS17,RS15} the focus is on undirected 
graphs with a \emph{hard} constraint on the edge density, $j$-star density
and triangle density, respectively. Following the work in Bhamidi {\it et al.}~\cite{BBS11}
and in Chatterjee and Diaconis~\cite{CD13}, a deeper understanding has developed of how these 
models behave as the size of the graph tends to infinity. Most results concern the asymptotic 
behaviour of the partition function (Chatterjee and Diaconis~\cite{CD13}, Kenyon, Radin, Ren 
and Sadun~\cite{KRRS17}) and the identification of regions where phase transitions 
occur (Aristoff and Zhu~\cite{AZpr}, Lubetsky and Zhao~\cite{LZ15}, Yin~\cite{Y13}). For more 
details we refer the reader to the recent monograph by Chatterjee~\cite{C15}, and references 
therein. Significant contributions for sparse graphs were made in Chatterjee and 
Dembo~\cite{CD16} and in subsequent work of Yin and Zhu~\cite{YZ17}. 

For an overview on random graphs and their role as models of complex networks, we refer the 
reader to the recent monograph by van der Hofstad~\cite{vdH17}. The most important distinction 
between  our paper and the existing literature on exponential random graphs is that in the 
canonical ensemble we impose a \emph{soft} constraint.

%%%

\subsection{Outline}
\label{S1.2}

The remainder of this paper is organised as follows. Section~\ref{S1alt} defines the two ensembles, 
gives the definition of equivalence of ensembles in the dense regime, recalls some basic facts 
about graphons, and states the large deviation principle for the Erd\H{o}s-R\'enyi random graph. 
Section~\ref{S2} states a key theorem in which we give a \emph{variational representation} of 
$s_\infty$ when the constraint is on \emph{subgraph counts}, properly normalised. Section~\ref{S3} 
presents our main theorem for ensemble equivalence, which provides three examples for which 
\emph{breaking of ensemble equivalence} occurs when the constraints are \emph{frustrated}. In 
particular, the constraints considered are on the number of edges, triangles and/or stars. Frustration corresponds to the situation where the canonical ensemble scales like an Erd\H{o}s-R\'enyi 
random graph model with an appropriate edge density but the microcanonical ensemble does not. 
The proof of the main theorem is given in Sections~\ref{S4}--\ref{S5}, and relies on various papers 
in the literature dealing with exponential random graph models. Appendix~\ref{app} discusses 
convergence of Lagrange multipliers associated with the canonical ensemble. 

%%%%%%%%%%%%%%%%%%%%%%%%%%%%%%%%%%%%%%%%%%%%%%

\section{Key notions}
\label{S1alt}

In Section~\ref{S1.3} we introduce the model and give our definition of equivalence of ensembles 
in the dense regime (Definition~\ref{def:ensequiv} below). In Section~\ref{S1.4} we recall some 
basic facts about graphons (Propositions~\ref{QSC}--\ref{Lip} below). In Section~\ref{S1.5} we recall 
the large deviation principle for the Erd\H{o}s-R\'enyi random graph (Proposition~\ref{RFprop} and 
Theorem~\ref{th:LDP} below), which is the key tool in our paper. 

\subsection{Microcanonical ensemble, canonical ensemble, relative entropy}
\label{S1.3}

For $n \in \N$, let $\cG_n$ denote the set of all $2^{{n\choose 2}}$ simple undirected graphs with 
$n$ vertices. Any graph $G\in\cG_n$ can be represented by a symmetric $n \times n$ matrix 
with elements 
\be
h^G(i,j) :=
\begin{cases}
1\qquad \mbox{if there is an edge between vertex } i \mbox{ and vertex } j,\\ 
0 \qquad \mbox{otherwise.}
\end{cases}
\ee
Let $\vC$ denote a vector-valued function on $\cG_n$. We choose a specific vector $\vC^*$,
which we assume to be \emph{graphic}, i.e., realisable by at least one graph in $\cG_n$.
For this  $\vC^*$ the \emph{microcanonical ensemble} is the probability distribution $\Pmic$ 
on $\cG_n$ with \emph{hard constraint} $\vC^*$ defined as
\begin{equation}
\Pmic(G) :=
\left\{
\begin{array}{ll} 
1/\Omega_{\vC^*}, \quad & \text{if } \vC(G) = \vC^*, \\ 
0, & \text{otherwise},
\end{array}
\right. \qquad G\in \cG_n,
\label{eq:PM}
\end{equation}
where 
\begin{equation}
\Omega_{\vC^*} := | \{G \in \cG_n\colon\, \vC(G) = \vC^* \} |
\end{equation}
is the number of graphs that realise $\vC^*$. The \emph{canonical ensemble} $\Pcan$
is the unique probability distribution on $\cG_n$ that maximises the \emph{entropy} 
\begin{equation}
S_n({\rm P}) := - \sum_{G \in \cG_n}{\rm P}(G) \log {\rm P}(G)
\end{equation}
subject to the \emph{soft constraint} $\langle \vC \rangle  = \vC^*$, where
\begin{equation}
\label{softconstr}
\langle \vC \rangle  := \sum_{G \in \cG_n} \vC(G)\,{\rm P}(G).
\end{equation}
This gives the formula (see Jaynes~\cite{J57})
\begin{equation}
\Pcan(G) := \frac{1}{Z(\vt^*)}\,\eee^{H(\vec{\theta}^*,\vec{C}(G))},
\qquad G \in \cG_n,
\label{eq:PC}
\end{equation}
with 
\be
H(\vec{\theta}^*,\vec{C}(G)) := \vt^* \cdot \vC(G), \qquad
Z(\vt^*\,) := \sum_{G \in \cG_n} \eee^{\vt^* \cdot\hspace{2pt} \vC(G)},
\label{eq:Ham}
\ee
denoting the \emph{Hamiltonian} and the \emph{partition function}, respectively. In 
\eqref{eq:PC}--\eqref{eq:Ham} the parameter $\vt^*$ (which is a real-valued vector
the size of the constraint playing the role of a Langrange multiplier) must be set to 
the unique value that realises $\langle \vC \rangle  = \vC^*$. The Lagrange multiplier 
$\vec{\theta}^*$ exists and is unique. Indeed, the gradients of the constraints in 
\eqref{softconstr} are linearly independent vectors. Consequently, the Hessian matrix 
of the entropy of the canonical ensemble in \eqref{eq:PC} is a positive definite matrix, 
which implies uniqueness of the Lagrange multiplier. 

The \emph{relative entropy} of $\Pmic$ with respect to  $\Pcan$ is defined as
\begin{equation}
S_n(\Pmic \mid \Pcan) 
:= \sum_{G \in \cG_n} \Pmic(G) \log \frac{\Pmic(G)}{\Pcan(G)}.
\label{eq:KL1}
\end{equation}

\begin{definition}
\label{def:ensequiv}
{\rm In the dense regime, if
\footnote{In Squartini \emph{et al.}~\cite{SdMdHG15}, which was concerned with the \emph{sparse 
regime}, the relative entropy was divided by $n$ (the number of vertices). In the \emph{dense regime}, 
however, it is appropriate to divide by $n^2$ (the order of the number of edges).}
\begin{equation}
\label{eq:Equivalence}
s_{\infty} := \lim_{n\to\infty} \frac{1}{n^2}\,S_{n}(\Pmic | \Pcan) = 0,
\end{equation}
then $\Pmic$ and $\Pcan$ are said to be \emph{equivalent}.}
\hfill \qed
\end{definition}

Before proceeding, we recall an important observation made in Squartini \emph{et al.}~\cite{SdMdHG15}. 
For any $G_1,G_2\in\cG_n$, $\Pcan(G_1)=\Pcan(G_2)$ whenever $\vC(G_1)=\vC(G_2)$, i.e., the 
canonical probability is the same for all graphs with the same value of the constraint. We may therefore 
rewrite \eqref{eq:KL1} as
\begin{equation}
S_n(\Pmic \mid \Pcan) = \log \frac{\Pmic(G^*)}{\Pcan(G^*)},
\label{eq:KL2}
\end{equation}
where $G^*$ is \emph{any} graph in $\cG_n$ such that $\vC(G^*) =\vC^*$ (recall that we 
assumed that $\vC^*$ is realisable by at least one graph in $\cG_n$).  This fact greatly 
simplifies computations. 

\begin{remark}
\label{rem:ndep}
{\rm All the quantities above depend on $n$. In order not to burden the notation, we exhibit 
this $n$-dependence only in the symbols $\cG_n$ and $S_n(\Pmic \mid \Pcan)$. When we 
pass to the limit $n\to\infty$, we need to specify how $\vC(G)$, $\vC^*$ and $\vec{\theta}^*$ 
are chosen to depend on $n$. This will be done in Section~\ref{S2.1}.} \hfill \qed
\end{remark}

%%%

\subsection{Graphons}
\label{S1.4}

There is a natural way to embed a simple graph on $n$ vertices in a space of functions called 
\emph{graphons}. Let $W$ be the space of functions $h\colon\,[0,1]^2 \to [0,1]$ such that 
$h(x,y) = h(y,x)$ for all $(x,y) \in [0,1]^2$. A finite simple graph $G$ on $n$ vertices can be
represented as a graphon $h^{G} \in W$ in a natural way as (see Fig.~\ref{fig-graphon})
\begin{equation}
\label{graphondef}
h^{G}(x,y) := \left\{ \begin{array}{ll} 
1 &\mbox{if there is an edge between vertex } \lceil{nx}\rceil \mbox{ and vertex } \lceil{ny}\rceil,\\
0 &\mbox{otherwise}.
\end{array} 
\right.
\end{equation}

%%%%%%%%%%%%%%%%%%%%%%%%%%%%%%%%%%%%%%%
\begin{figure}[htbp]
\centering
\hspace{1.5cm}\includegraphics[width=0.7\linewidth, height=4.6cm]{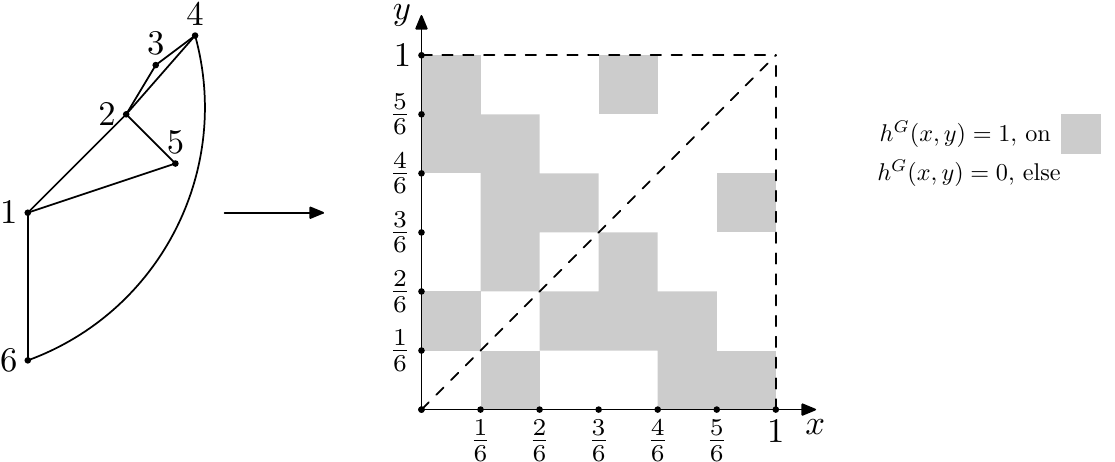}
\caption{\small An example of a graph $G$ and its graphon representation $h^G$.} 
\label{fig-graphon}
\end{figure}
%%%%%%%%%%%%%%%%%%%%%%%%%%%%%%%%%%%%%%%

\medskip\noindent
The space of graphons $W$ is endowed with the \emph{cut distance}
\begin{equation}
d_{\square} (h_1,h_2) := \sup_{S,T\subset [0,1]} 
\left|\int_{S\times T} \dd x\,\dd y\,[h_1(x,y) - h_2(x,y)]\right|,
\qquad h_1,h_2 \in W.
\end{equation}
On $W$ there is a natural equivalence relation $\equiv$. Let $\Sigma$ be the space of 
measure-preserving bijections $\sigma\colon\, [0,1] \to [0,1]$. Then $h_1(x,y)\equiv h_2(x,y)$ 
if $h_1(x,y) = h_2(\sigma x, \sigma y)$ for some $\sigma\in\Sigma$. This equivalence relation 
yields the quotient space $(\tilde{W},\delta_{\square})$, where $\delta_{\square}$ is the 
metric defined by 
\begin{equation}
\label{deltam}
\delta_{\square}(\tilde{h}_1,\tilde{h}_2) 
:= \inf _{\sigma_1,\sigma_2} d_{\square}(h_1^{\sigma_1}, h_2^{\sigma_2}),
\qquad \tilde{h}_1,\tilde{h}_2 \in \tilde{W}.
\end{equation}
To avoid cumbersome notation, throughout the sequel we suppress the $n$-dependence. Thus,
by $G$ we denote any simple graph on $n$ vertices, by $h^G$ its image in the graphon space 
$W$, and by $\tilde{h}^G$ its image in the quotient space $\tilde{W}$. Let $F$ and $G$ denote 
two simple graphs with vertex sets $V(F)$ and $V(G)$, respectively, and let $\text{hom}(F,G)$
be the number of homomorphisms from $F$ to $G$. The \emph{homomorphism density} is 
defined as 
\begin{equation}
t(F,G) := \frac{1}{|V(G)|^{|V(F)|}}\,\text{hom}(F,G).
\end{equation}
Two graphs are said to be \emph{similar} when they have similar homomorphism densities. 

\begin{definition}
\label{DefHom}
{\rm A sequence of labelled simple graphs $(G_n)_{n\in\N}$ is left-convergent when 
$(t(F,G_n))_{n\in\N}$ converges for any simple graph $F$.} \hfill \qed
\end{definition}
Consider a simple graph $F$ on $k$ vertices with edge set $E(F)$, and let $h \in W$. 
Similarly as above, define the density 
\begin{equation}
\label{graphon density}
t(F,h) := \int_{[0,1]^k} \dd x_1 \cdots \dd x_k \prod_{(i,j) \in E(F)} h(x_i,x_j).
\end{equation}
If $h^G$ is the image of a graph $G$ in the space $W$, then
\begin{equation}
\label{GGHOM}
t(F,h^G) = \int_{[0,1]^k} \dd x_1 \cdots \dd x_k \prod_{(i,j) \in E(F)} h^G(x_i,x_j) 
= \frac{1}{|V(G)|^{|V(F)|}}\,\text{hom}(F,G) = t(F,G).
\end{equation}
Hence a sequence of graphs $(G_n)_{n\in\N}$ is left-convergent to $h \in W$ when
\begin{equation}
\lim_{n\to\infty} t(F,G_n) = t(F,h).
\end{equation}

We conclude this section with three basic facts that will be needed later on. The first 
gives the relation between left-convergence of sequences of graphs and convergence 
in the quotient space $(\tilde{W},\delta_{\square})$, the second is a compactness 
property, while the third shows that the homomorphism density is Lipschitz continuous 
with respect to the $\delta_{\square}$-metric. 

\begin{proposition}[Borgs \textit{et al.}~\cite{BCLSV08}]
\label{QSC}
For a sequence of labelled simple graphs $(G_n)_{n\in\N}$ the following properties are 
equivalent:\\
(i) $(G_n)_{n\in\N}$ is left-convergent.\\
(ii) $(\tilde{h}^{G_n})_{n\in\N}$ is a Cauchy sequence in the metric $\delta_{\square}$.\\
(iii) $(t(F, h^{G_n}))_{n\in\N}$ converges for all finite simple graphs $F$.\\ 
(iv) There exists an $h\in W$ such that $\lim_{n\to\infty} t(F, h^{G_n}) = t(F,h)$ for all finite 
simple graphs $F$.  
\end{proposition}

\begin{proposition}[Lov\'asz and Szegedy \cite{LS06}]
\label{comp}
$(\tilde{W}, \delta_{\square})$ is compact.  
\end{proposition}

\begin{proposition}[Borgs \textit{et al.}~\cite{BCLSV08}]
\label{Lip}
Let $G_1,G_2$ be two labelled simple graphs, and let $F$ be a simple graph. Then 
\begin{equation}
|t(F,G_1) - t(F,G_2)| \leq 4|E(F)|\delta_{\square}(G_1,G_2).
\end{equation}
\end{proposition}

For a more detailed description of the structure of the space $(\tilde{W},\delta_{\square})$ 
we refer the reader to Borgs \textit{et al.}~\cite{BCLSV08, BCLSV12} and Diao 
\textit{et at.}~\cite{DGKR15}.

%%%

\subsection{Large deviation principle for the Erd\H{o}s-R\'enyi random graph}
\label{S1.5}

In this section we recall a few key facts from the literature about rare events in Erd\H{o}s-R\'enyi 
random graphs, formulated in terms of a large deviation principle. Importantly, the scale that is used 
is $n^2$, the order of the number of \emph{edges} in the graph.

We start by introducing the large deviation rate function. For $p\in(0,1)$ and $u\in[0,1]$, 
let 
\begin{equation}
\label{eq: rate function}
\begin{aligned}
I_p(u) &:= \frac{1}{2}u\log\left(\frac{u}{p}\right) + \frac{1}{2}(1-u)\log\left(\frac{1-u}{1-p}\right),\\
I(u) &:= \frac{1}{2}u\log u +\frac{1}{2}(1-u)\log(1-u) = I_{\tfrac12}(u) - \tfrac12 \log 2,
\end{aligned}
\end{equation}
with the convention that $0\log0=0$. For $h\in W$ we write, with a mild abuse of notation,  
\begin{equation}
\label{eq: rate 2}
I_p(h) := \int_{[0,1]^2} \dd x\, \dd y\,\,I_p(h(x,y)), \qquad
I(h) := \int_{[0,1]^2} \dd x\, \dd y\,\,I(h(x,y)).
\end{equation}
On the quotient space $(\tilde{W},\delta_{\square})$ we define $I_p(\tilde{h}) = I_p (h)$, where 
$h$ is any element of the equivalence class $\tilde{h}$. 

\begin{proposition}[Chatterjee and Varadhan~\cite{CV11}]
\label{RFprop}
The function $I_p$ is well-defined on $\tilde{W}$ and is lower semi-continuous under the 
$\delta_{\square}$-metric.  
\end{proposition}

Consider the set $\cG_n$ of all graphs on $n$ vertices and the Erd\H{o}s-R\'enyi 
probability distribution $\PP_{n,p}$ on $\cG_n$. Through the mappings $G \to h^{G} \to 
\tilde{h}^G$ we obtain a probability distribution on $W$ (with a slight abuse of notation again 
denoted by $\PP_{n,p}$), and a probability distribution $\tilde{\PP}_{n,p}$ on $\tilde{W}$. 

\begin{theorem}[Chatterjee and Varadhan~\cite{CV11}]
\label{th:LDP}
For every $p \in (0,1)$, the sequence of probability distributions $(\tilde{\PP}_{n,p})_{n\in\N}$ 
satisfies the large deviation principle on $(\tilde{W},\delta_{\square})$ with rate function $I_p$ 
defined by \eqref{eq: rate 2}, i.e.,
\begin{equation}
\begin{aligned}
\limsup_{n\to\infty} \frac{1}{n^2}\, \log\tilde{\PP}_{n,p}(\tilde{C}) 
&\leq -\inf_{\tilde{h}\in\tilde{W}}I_p(\tilde{h}) \qquad \forall\,\tilde{C}\subset \tilde{W} \mbox{ closed},\\
\liminf_{n\to\infty}\ \frac{1}{n^2}\, \log\tilde{\PP}_{n,p}(\tilde{O})
&\geq -\inf_{\tilde{h}\in\tilde{O}}I_p(\tilde{h}) \qquad \,\,\,\forall\,\tilde{O}\subset \tilde{W} \mbox{ open}.
\end{aligned}
\end{equation}

\end{theorem} 

Using the large deviation principle we can find asymptotic expressions for the number of simple 
graphs on $n$ vertices with a given property. In what follows a property of a graph is defined 
through an \emph{operator} $T\colon\,W\to\R^m$ for some $m \in \N$. We assume that the operator 
$T$ is continuous with respect to the $\delta_{\square}$-metric, and for some $\vec{T}^*\in\R^m$ we
consider the sets 
\begin{equation} 
\label{W*def}
\tilde{W}^* := \big\{\tilde{h}\in\tilde{W}\colon\,T(\tilde{h}) = \vec{T}^*\big\},
\qquad \tilde{W}^*_n := \big\{\tilde{h}\in \tilde{W}^*\colon\,\tilde{h}=\tilde{h}^G 
\text{ for some } G \text{ on } n \text{ vertices}\big\}.
\end{equation}
By the continuity of the operator $T$, the set $\tilde{W}^*$ is closed. Therefore, using 
Theorem~\ref{th:LDP}, we obtain the following asymptotics for the cardinality of $\tilde{W}^*_n$.

\begin{corollary}[Chatterjee~\cite{C16}]
\label{Cor: Count}
For any measurable set $\tilde{W}^*\subset\tilde{W}$, with $\tilde{W}^*_n$ as defined in 
\eqref{W*def},
\begin{equation}
-\inf_{\tilde{h} \in {\rm int}(\tilde{W}^*)}I(\tilde{h}) \leq 
\liminf_{n\to\infty} \frac{\log|\tilde{W}^*_n|}{n^2} \leq
\limsup_{n\to\infty} \frac{\log|\tilde{W}^*_n|}{n^2} \leq
-\inf_{\tilde{h} \in \tilde{W}^*}I(\tilde{h}),
\end{equation}
where ${\rm int}(\tilde{W}^*)$ is the interior of $\tilde{W}^*$. 
\end{corollary}

%%%%%%%%% SECTION 2 %%%%%%%%%%%%%%%%%%%%

\section{Variational characterisation of ensemble equivalence}
\label{S2}

In this section we present a number of preparatory results we will need in Section~\ref{S3}
to state our theorem on the equivalence between $\Pmic$ and $\Pcan$. Our main result 
is Theorem~\ref{th:Limit} below, which gives us a \emph{variational characterisation of ensemble 
equivalence}. In Section~\ref{S2.1} we introduce our constraints on the subgraph counts. In 
Section~\ref{S2.2} we rephrase the canonical ensemble in terms of graphons.  In Section~\ref{S2.3} 
we state and prove Theorem~\ref{th:Limit}.  

%%%

\subsection{Subgraph counts}
\label{S2.1}

First we introduce the concept of subgraph counts, and point out how the corresponding 
canonical distribution is defined. Label the simple graphs in any order, e.g., $F_1$ is an 
edge, $F_2$ is a wedge, $F_3$ is triangle, etc. Let $C_k(G)$ denote the number of 
subgraphs $F_k$ in $G$. In the dense regime, $C_k(G)$ grows like $n^{V_k}$, where 
$V_k=|V(F_k)|$ is the number of vertices in $F_k$. For $m \in \N$, consider the following 
\emph{scaled vector-valued function} on $\cG_n$:
\begin{equation}
\vec{C}(G) := \left(\frac{p(F_k)C_{k}(G)}{n^{V_k-2}}\right)_{k=1}^m
= n^2\left(\frac{p(F_k)C_{k}(G)}{n^{V_k}}\right)_{k=1}^m.
\end{equation}
The term $p(F_k)$ counts the edge-preserving permutations of the vertices of $F_k$, i.e., 
$p(F_1)=2$ for an edge, $p(F_2)=2$ for a wedge, $p(F_3)=6$ for a triangle, etc. 
The term $C_k(G)/n^{V_k}$ represents a subgraph density in the graph $G$. The additional 
$n^2$ guarantees that the full vector scales like $n^2$, the scaling of the large deviation principle in 
Theorem \ref{th:LDP}. For a simple graph $F_k$ we define 
the \emph{homomorphism density} as
\begin{equation}
t(F_k,G) := \frac{\text{hom}(F_k,G)}{n^{V_k}}
= \frac{p(F_k)C_k(G)}{n^{V_k}},
\end{equation}
which does not distinguish between permutations of the 
vertices. Hence the Hamiltonian becomes
\begin{equation}
\label{eq:HF}
H(\vec{\theta}, \vec{T}(G))=n^2 \sum_{k=1}^m \theta_k \,t(F_k,G) 
= n^2 (\vec{\theta}\cdot\vec{T}(G)), \qquad G \in \cG_n,
\end{equation}
where 
\begin{equation}
\label{operator}
\vec{T}(G) := \left(t(F_k, G)\right)_{k=1}^m.
\end{equation}
The canonical ensemble with parameter $\vec{\theta}$ thus takes the form
\begin{equation}
\label{eq:CPD}
\Pcan(G \mid \vec{\theta}\,) := \eee^{n^2\big[\vec{\theta}\cdot\vec{T}(G)-\psi_n(\vec{\theta}\,)\big]},
\qquad G \in \cG_n,
\end{equation}
where $\psi_n$ replaces the \emph{partition function}: 
\begin{equation}
\label{eq:PF}
\psi_n(\vec{\theta}) := \frac{1}{n^2}\log\sum_{G\in\cG_n} 
\eee^{n^2 (\vec{\theta}\hspace{2pt}\cdot\hspace{2pt}\vec{T}(G))}.
\end{equation}
In the sequel we take $\vec{\theta}$ equal to a specific value $\vec{\theta}^*$, so as to meet 
the soft constraint, i.e.,
\begin{equation}
\label{softconstraint2}
\langle \vec{T} \rangle   = \sum_{G\in\cG_n}\vec{T}(G)\,\Pcan(G) = \vec{T}^*.
\end{equation}
 The canonical probability then becomes
\begin{equation}
\label{canonicprob}
\Pcan(G) = \Pcan(G\mid \vec{\theta^*})
\end{equation}
In Section~\ref{S4.1} we will discuss how to find $\vec{\theta}^*$. 

\begin{remark}
\label{rem:LMn}
{\rm (i) The constraint $\vec{T}^*$ and the Lagrange multiplier $\vec{\theta}^*$ in general depend on $n$, 
i.e., $\vec{T}^*=\vec{T}^*_n$ and $\vec{\theta}^* = \vec{\theta}^*_n$ (recall Remark~\ref{rem:ndep}). 
We consider constraints that converge when we pass to the limit $n\to\infty$, i.e., 
\begin{equation}
\label{eq: Assumption T}
\lim_{n\to\infty} \vec{T}^*_n = \vec{T}^*_\infty.
\end{equation}
Consequently, we expect that  
\begin{equation}
\label{eq:Assumption}
\lim_{n\to\infty} \vec{\theta}^*_n = \vec{\theta}^*_\infty.
\end{equation}
Throughout the sequel we \emph{assume} that \eqref{eq:Assumption} holds. If convergence fails, 
then we may still consider subsequential convergence. The subtleties concerning \eqref{eq:Assumption} 
are discussed in Appendix~\ref{app}.\\ 
(ii) In what follows, we suppress the dependence on $n$ and write $\vec{T}^*,\vec{\theta}^*$ instead 
of $\vec{T}^*_n,\vec{\theta}^*_n$, but we keep the notation $\vec{T}^*_\infty,\vec{\theta}^*_{\infty}$ 
for the limit. In addition, throughout the sequel we write $\vec{\theta},\vec{\theta}_{\infty}$ instead of 
$\vec{\theta}^*,\vec{\theta}_{\infty}^*$ when we view these as parameters that do not depend on $n$. 
This distinction is crucial when we take the limit $n\to\infty$.}
\hfill\qed
\end{remark}

%%%

\subsection{From graphs to graphons}
\label{S2.2}

In \eqref{GGHOM} we saw that if we map a finite simple graph $G$ to its graphon $h^G$, then 
for each finite simple graph $F$ the homomorphism densities $t(F,G)$ and $t(F,h^G)$ are 
identical. If $(G_n)_{n\in\N}$ is left-convergent, then 
\begin{equation}
\lim_{n\to\infty} \vec{T} (G_n) = \left(t(F_k, h)\right)_{k=1}^m
\end{equation}
for some $h \in W$, as an immediate consequence of Theorem~\ref{QSC}. We further see that 
the expression in \eqref{eq:HF} can be written in terms of graphons as
\begin{equation}
\label{eq:HFG}
H(\vec{\theta}, \vec{T}(G))=n^2 \sum_{k=1}^m \theta_k \,t(F_k,h^G).
\end{equation}
With this scaling the \emph{hard constraint} is denoted by $\vec{T}^*$, has the interpretation of 
the \emph{density} of an observable quantity in $G$, and defines a subspace of the quotient space 
$\tilde{W}$, which we denote by $\tilde{W}^*$, and which consists of all graphons that meet the 
hard constraint, i.e., 
\begin{equation}
\tilde{W}^* := \{\tilde{h}\in \tilde{W}\colon\,\vec{T}(h) = \vec{T}^*\}.
\end{equation}
The \emph{soft constraint} in the canonical ensemble becomes $\langle \vec{T} \rangle 
= \vec{T}^*$ (recall \eqref{softconstr}).

%%%
\vspace{-2mm}

%%%

\subsection{Variational formula for specific relative entropy}
\label{S2.3}

In what follows, the limit as $n \to \infty$ of the partition function $\psi_n(\vec{\theta})$ defined in 
(\ref{eq:PF}) plays an important role. This limit has a variational representation that will be key to 
our analysis.

\begin{theorem}[Chatterjee and Diaconis \cite{CD13}] 
\label{thm: CD partition function}
Let $\vec{T}\colon\,\tilde{W}\rightarrow\mathbb{R}^m$be the operator defined in \eqref{operator}. 
For any $\vec{\theta}\in\mathbb{R}^m$ (not depending on $n$),
\begin{equation}
\label{eq: CD variational formula}
\lim_{n\rightarrow\infty}\psi_n(\vec{\theta}) 
= \sup_{\tilde{h}\in\tilde{W}}\left(\vec{\theta}\cdot \vec{T}(\tilde{h}) -I(\tilde{h})\right)
\end{equation}
with $I$ and $\psi_n$ as defined in \eqref{eq: rate 2} and \eqref{eq:PF}.
\end{theorem}

\begin{theorem}[Chatterjee and Diaconis \cite{CD13}] 
\label{thm: CD variational simplify}
Let $F_1,\hdots,F_m$ be subgraphs as defined in Section~{\rm \ref{S2.1}}. 
Suppose that $\theta_2,\hdots,\theta_m \geq 0$. Then
\begin{equation}
\label{eq: CD variational simplify}
\lim_{n\rightarrow\infty}\psi_n(\vec{\theta}) 
= \sup_{0\leq u\leq 1} \left(\sum_{i=1}^m\theta_i\, u^{E(F_k)} - I(u)\right),
\end{equation}
where $E(F_k)$ denotes the number of edges in the subgraph $F_k$.
\end{theorem}

The key result in this section is the following variational formula for $s_\infty$ defined 
in Definition~\ref{def:ensequiv}. Recall that for $n\in\N$ we write $\vec{\theta}^*$ for 
$\vec{\theta}^*_n$.

\begin{theorem}
\label{th:Limit}
Consider the microcanonical ensemble defined in \eqref{eq:PM} with constraint 
$\vec{T}=\vec{T}^*$ defined in \eqref{operator}, and the canonical ensemble defined in 
\eqref{eq:CPD}--\eqref{eq:PF} with parameter $\vec{\theta}=\vec{\theta}^*$ such that, 
for every $n\in\mathbb{N}$, \eqref{softconstraint2}, \eqref{eq: Assumption T} and \eqref{eq:Assumption} hold. Then
\begin{equation}
\label{varreprsinfty}
s_\infty = \lim_{n\to\infty} \frac{1}{n^2}\,S_n(\Pmic \mid \Pcan) 
= \sup_{\tilde{h}\in \tilde{W}} \big[\vec{\theta}^*_\infty\cdot\vec{T}(\tilde{h})-I(\tilde{h})\big]
-\sup_{\tilde{h}\in \tilde{W}^*} \big[\vec{\theta}^*_\infty\cdot\vec{T}(\tilde{h}) - I(\tilde{h})\big],
\end{equation}
where $I$ is defined in \eqref{eq: rate function} and $\tilde{W}^*=\{\tilde{h}\in\tilde{W}
\colon\,\vec{T}(\tilde{h}) = \vec{T}^*_\infty\}$.  
\end{theorem}

\begin{proof}
From \eqref{eq:KL2} we have
\begin{equation}
\label{eq:SPrel}
s_{\infty} = \lim_{n\to\infty} \frac{1}{n^2} \big[\log\,\Pmic(G^*) -\log\, \Pcan(G^*)\big],
\end{equation}
where $G^*$ is any graph in $\cG_n$ such that $\vec{T}(G^*) = \vec{T}^*$. For the 
microcanonical ensemble we have
\begin{equation}
\label{eq: MIC}
\log\Pmic(G^*) = - \log\Omega_{\vec{T}^*}
= -\log\PP_{\frac{1}{2},n}\left(\{G\in\cG_n\colon\, \vec{T}(G) 
= \vec{T}^*\}\right) - {n\choose2}\log2,
\end{equation}
where 
\begin{equation}
\Omega_{\vec{T}^*} = | \{G \in \cG_n\colon\, \vec{T}(G) = \vec{T}^* \} | > 0.
\end{equation}
Define the operator $\vec{T}\colon\,W \to \R^m$, $h\mapsto (t(F_k,h))_{k=1}^m$. This operator 
can be extended to an operator (with a slight abuse of notation again denoted by $\vec{T}$) 
on the quotient space $(\tilde{W},\delta_{\square})$ by defining $\vec{T}(\tilde{h})=\vec{T}(h)$ 
with $h\in\tilde{h}$. Define the following sets
\begin{equation} 
\tilde{W}^* := \big\{\tilde{h}\in\tilde{W}\colon\,T(\tilde{h}) = \vec{T}^*_\infty\big\}, \qquad 
\tilde{W}^*_n :=\big\{\tilde{h}\in \tilde{W}^*\colon\, \tilde{h}=\tilde{h}^G \text{ for some } 
G \in \cG_n\big\}.
\end{equation}
From the continuity of the operator $\vec{T}$ on $\tilde{W}$, we see that $\tilde{W}^*$ is a 
compact subspace of $\tilde{W}$, and hence is also closed. From Theorem~\ref{Lip} we have 
that $\vec{T}$ is a Lipschitz continuous operator on the space $(\tilde{W},\delta_{\square})$. 
Since $\tilde{W}$ is a compact space, we have that 
\begin{equation}
\limn\frac{1}{n^2}\,\log\PP_{\frac{1}{2},n}\left(\{G\in\cG_n\colon\, \vec{T}(G) = \vec{T}^*\}\right) 
= -\inf_{\tilde{h}\in\tilde{W}^*} I_{\frac{1}{2}}(\tilde{h})
= -\inf_{\tilde{h}\in\tilde{W}^*} I(\tilde{h})-\tfrac12\log2.
\end{equation}
The large deviation principle applied to \eqref{eq: MIC} yields 
\begin{equation}
\label{eq: MIC2}
\limn \frac{1}{n^2}\,\log\Pmic(G^*) = \inf_{\tilde{h}\in\tilde{W}^*} I(\tilde{h}).
\end{equation}

Consider the canonical ensemble and a graph $G_n^*$ on $n$ vertices such that $\vec{T}(G_n^*) 
= \vec{T}^*$. By Definition \ref{DefHom}, Proposition \ref{QSC}, and \eqref{eq: Assumption T} we may 
suppose that $(G_n^*)_{n\in\N}$ is left-convergent and converges to the graphon $h^*$. Since $\vec{T}$ 
is continuous, we have that $\vec{T}(G_n^*)$ converges to $\vec{T}(h^*)=\vec{T}^*_\infty$. From 
\eqref{eq:CPD} we have that 
\begin{equation}
\limn\frac{1}{n^2}\, \log\Pcan(G_n^*) = \vec{\theta}^*_\infty\cdot\vec{T}^*_\infty 
- \psi_\infty(\vec{\theta}^*_{\infty}).
\end{equation}
By Theorem \ref{thm: CD partition function},
\begin{equation}
\label{eq: CAN2}
\psi_\infty(\vec{\theta}^*_{\infty})
= \sup_{\tilde{h}\in \tilde{W}} \big[\vec{\theta}^*_\infty\cdot\vec{T}(\tilde{h})-I(\tilde{h})\big].
\end{equation}
There is an additional subtlety in proving \eqref{eq: CAN2} in our setup because $\vec{\theta}^*$ 
depends on $n$. This dependence is treated in Appendix A. Combining \eqref{eq: MIC2} and 
\eqref{eq: CAN2}, we get 
\begin{equation}
\label{eq: CAN3}
s_\infty = \lim_{n\to\infty} \frac{1}{n^2}\, S_n(\Pmic \mid \Pcan) 
= \inf_{\tilde{h}\in\tilde{W}^*} I(\tilde{h})-\vec{\theta}^*_\infty\cdot\vec{T}^*_\infty
+ \sup_{\tilde{h}\in \tilde{W}} \big[\vec{\theta}^*_\infty\cdot\vec{T}(\tilde{h})-I(\tilde{h})\big].
\end{equation}
By definition all elements $\tilde{h}\in\tilde{W}^*$ satisfy $\vec{T}(\tilde{h}) = \vec{T}^*_\infty$.
Hence the expression in the right-hand side of (\ref{eq: CAN3}) can be written as 
\begin{equation}
\sup_{\tilde{h}\in \tilde{W}} \big[\vec{\theta}^*_\infty\cdot\vec{T}(\tilde{h})-I(\tilde{h})\big]
-\sup_{\tilde{h}\in \tilde{W}^*}\big[\vec{\theta}^*_\infty\cdot\vec{T}(\tilde{h}) - I(\tilde{h})\big],
\end{equation}
which settles the claim.
\end{proof}

\begin{remark}
{\rm Theorem~\ref{th:Limit} and the compactness of $\tilde{W}^*$ give us a \emph{variational 
characterisation} of ensemble equivalence: $s_\infty = 0$ if and only if at least one of the 
maximisers of $\vec{\theta}^*_\infty\cdot\vec{T}(\tilde{h})-I(\tilde{h})$ in $\tilde{W}$ also lies 
in $\tilde{W}^* \subset \tilde{W}$. Equivalently, $s_{\infty}=0$ when at least one the maximisers 
of $\vec{\theta}^*_\infty\cdot\vec{T}(\tilde{h})-I(\tilde{h})$ satisfies the hard constraint.} \hfill\qed
\end{remark}

%%%%%%%% SECTION 3 %%%%%%%%%%%%%%%%%%%%%%

\section{Main theorem}
\label{S3}

The variational formula for the relative entropy $s_\infty$ in Theorem~\ref{th:Limit} allows us to 
identify examples where ensemble equivalence holds ($s_\infty=0$) or is broken ($s_\infty>0$). 
We already know that if the constraint is on the edge density alone, i.e., $T(G) = t(F_1,G) = T^*$, 
then $s_\infty=0$ (see Garlaschelli \emph{et al.}~\cite{GHR17}). In what follows we will look at three 
models: 

%%%%%%%%%%%%%%%%%%%%%%%%%%%%%%%%%%%%%%%%%%%%%
\begin{figure}[htbp]
\centering
\hspace{1.5cm}\includegraphics[width=0.7\linewidth, height=4.6cm]{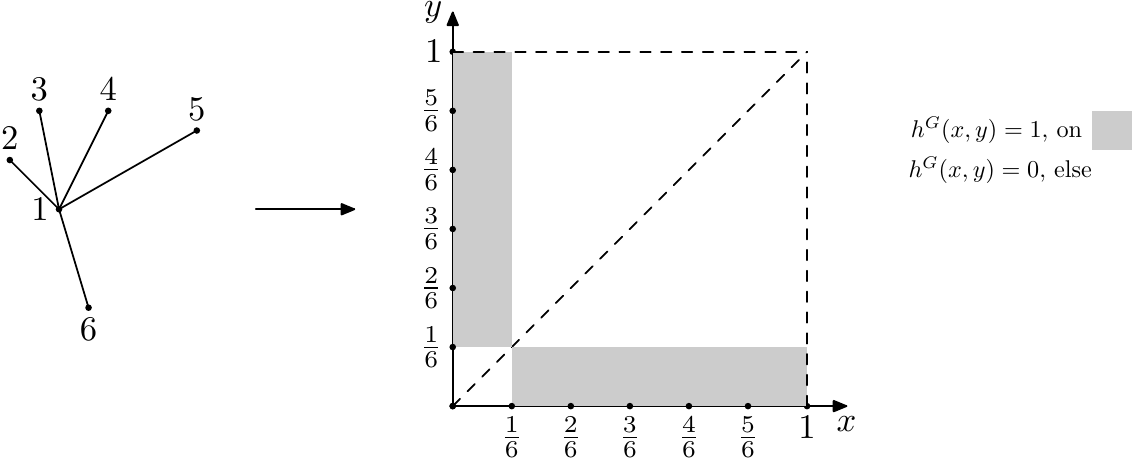}
\caption{\small A 5-star graph and its graphon representation.} 
\label{fig-5star}
\end{figure}
%%%%%%%%%%%%%%%%%%%%%%%%%%%%%%%%%%%%%%%%%%%%

\begin{itemize}
\item[(I)]
The constraint is on the triangle density, i.e., $\vec{T_2}(G) = t(F_3,G) = T_2^*$ with 
$F_3$ the triangle. This will be referred to as the \emph{Triangle Model}.
\item[(II)]
The constraint is on the edge density and triangle density, i.e., $\vec{T}(G) = (t(F_1,G),$
$t(F_3,G)) = (T_1^*,T_2^*)$ with $F_1$ the edge and $F_3$ the triangle. This will be referred 
to as the \emph{Edge-Triangle Model}.
\item[(III)]
The constraint is on the $j$-star density, i.e., $\vec{T}(G) = t(T[j],G)=T[j]^*$ with $T[j]$ the $j$-star 
graph, consisting of 1 root vertex and $j \in \N\setminus\{1\}$ vertices connected to the root but 
not connected to each other (see Fig.~\ref{fig-5star}). This will be referred to as the \emph{Star 
Model}.
\end{itemize}

\noindent
For a graphon $h\in W$ (recall \eqref{graphon density}), the edge density and the triangle density 
equal
\begin{equation}
\label{eq: densityalt}
T_1(h) = \int_{[0,1]^2} \dd x_1 \dd x_2\, h(x_1,x_2), \quad 
T_2(h)= \int_{[0,1]^3} \dd x_1 \dd x_2\dd x_3\, h(x_1,x_2)h(x_2,x_3)h(x_3,x_1),
\end{equation} 
while the $j$-star density equals
\begin{equation}
\label{eq: densitystar}
T[j](h) = \int_{[0,1]} \dd x \int_{[0,1]^j} \dd x_1 \dd x_2 \cdots \dd x_j\, 
\prod_{i=1}^j h(x,x_i).
\end{equation}

\begin{theorem} 
\label{thm:equivalence}
For the above three types of constraint: 
\begin{itemize}
\item[{\rm (I)}] 
\begin{itemize}
\item[{\rm (a)}]
If $T_2^*\geq \tfrac18$ , then $s_{\infty} = 0$.
\item[{\rm (b)}]
If $T_2^* = 0$, then $s_{\infty} = 0$. 
\end{itemize}
\item[{\rm (II)}]
\begin{itemize}
\item[{\rm (a)}]
If $T_2^* = T_1^{*{3}}$, then $s_{\infty} = 0$.
\item[{\rm (b)}]
If $T_2^* \neq T_1^{*{3}}$ and $T_2^* \geq \tfrac18$, then $s_{\infty} > 0$.
\item[{\rm (c)}]
If $T_2^*\neq T_1^{*{3}}$, $0<T_1^*\leq\tfrac12$ and $0<T_2^*<\frac{1}{8}$, 
then $s_{\infty}>0$. 
\item[{\rm (d)}]
If $T_1^* = \tfrac12 +\epsilon$ with $\epsilon\in\left(\frac{\ell-2}{2\ell},\frac{\ell-1}{2\ell+2}\right)$, 
$\ell \in \N\setminus\{1\}$, and $T_2^*$ is such that $(T_1^*,T_2^*)$ lies on the scallopy curve
in Fig.~{\rm \ref{fig-scallopy}}, then $s_{\infty}>0$. 
\item[{\rm (e)}] 
If $0<T_1^* \leq \frac{1}{2}$ and $T_2^* = 0$, then $s_{\infty}=0$.
\end{itemize}
\item[{\rm (III)}]
For every $j \in \N\setminus \{1\}$, if $T[j]^* \geq 0$, then $s_{\infty}=0$.
\end{itemize}

\end{theorem} 
Here, $T_1^*,T_2^*,T[j]^*$ are in fact the limits $T_{1,\infty}^*,T_{2,\infty}^*,T[j]^*_\infty$ in 
\eqref{eq: Assumption T}, but in order to keep the notation light we now also suppress the 
index $\infty$.

%%%%%%%%%%%%%%%%%%%%%%%%%%%%%%%%%%%%%%%%%
\begin{figure}[htbp]
\centering
\hspace{1.5cm}\includegraphics[width=0.7\linewidth,height=7cm]{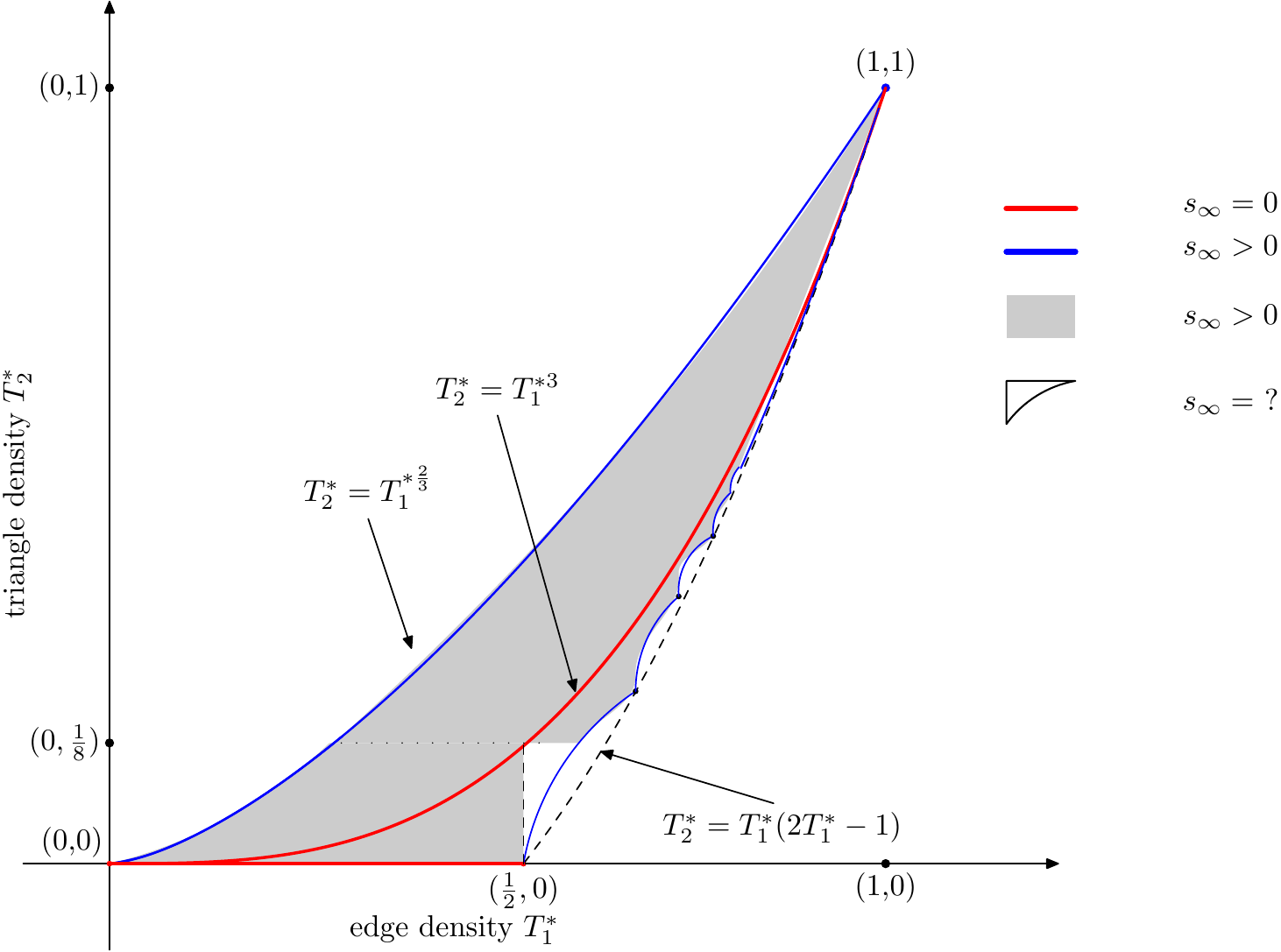}
\caption{\small The admissible edge-triangle density region is the region on and 
between the blue curves (cf.\ Radin and Sadun~\cite{RS15}).}
\label{fig-scallopy}
\end{figure}
%%%%%%%%%%%%%%%%%%%%%%%%%%%%%%%%%%%%%%%%%

\noindent
Theorem~\ref{thm:equivalence}, which states our main results on ensemble equivalence 
and which is proven in Sections~\ref{S4}--\ref{S5}, is illustrated in Fig.~\ref{fig-scallopy}. 
The region on and between the blue curves corresponds to the set of all realisable graphs: 
if the pair $(e,t)$ lies in this region, then there exists a graph with edge density $e$ and 
triangle density $t$. The red curves represent ensemble equivalence, the blue curves and 
the grey region represent breaking of ensemble equivalence, while in the white region 
between the red curve and the lower blue curve we do not know what happens.
Breaking of ensemble equivalence arises from frustration between the edge and the triangle density. 

Each of the cases in Theorem~\ref{thm:equivalence} corresponds to typical behaviour of
graphs drawn from the two ensembles: 
\begin{itemize}
\item
In cases (I)(a) and (II)(a), graphs drawn from both ensembles are asymptotically like Erd\H{o}s-R\'enyi 
random graphs with parameter $p=T_2^{*1/3}$.
\item
In cases (I)(b) and (II)(e), almost all graphs drawn from both ensembles are asymptotically like bipartite graphs. 
\item
In cases (II)(b), (II)(c) and (II)(d), we do not know what graphs drawn from the canonical ensemble look like.
Graphs drawn from the microcanonical ensemble do not look like Erd\H{o}s-R\'enyi random graphs. 
The structure of graphs drawn from the microcanonical ensemble when the constraint is as in (II)(d) 
has been determined in Pirkhurko and Razborov~\cite{PR12} and Radin and Sadun~\cite{RS15}. The vertex 
set of a graph drawn from the microcanonical ensemble can be partitioned into $\ell$ subsets: the first 
$\ell-1$ have size $\lfloor cn \rfloor$ and the last has size between $\lfloor cn\rfloor$ and $2\lfloor cn \rfloor$, 
where $c$ is a known constant depending on $\ell$. The graph has the form of a complete $\ell-$partite 
graph on these pieces, plus some additional edges in the last piece that create no additional triangles. 
\item
In case (III), graphs drawn from both ensembles are asymptotically like Erd\H{o}s-R\'enyi random graphs 
with parameter $p = T[j]^{*1/j}$.
\end{itemize}

\begin{remark}
{\rm Similar results hold for the Edge-Wedge-Triangle Model and the Edge-Star Model.} 
\hfill \qed
\end{remark}

Here are three open questions:
\begin{itemize}
\item
Identify in which cases \eqref{eq: Assumption T} implies \eqref{eq:Assumption}.
\item
Is $s_\infty=0$ as soon as the constraint involves a \emph{single} subgraph count only? 
\item
What happens for subgraphs other than edges, wedges, triangles and stars? Is again
$s_\infty>0$ under appropriate frustration?
\end{itemize}

%%%%%%% SECTION 4 %%%%%%%%%%%%%%%%%%%%

\section{Choice of the tuning parameter}
\label{S4}

The tuning parameter is to be chosen so as to satisfy the soft constraint \eqref{softconstraint2}, 
a procedure that in equilibrium statistical physics is referred to as the \emph{averaging principle}. 
Depending on the choice of constraint, finding $\vec{\theta}^*$ may not be easy, neither analytically 
nor numerically. In Section~\ref{S4.1} we investigate how $\vec{\theta}^*$ behaves as we vary 
$\vec{T}^*$ for fixed $n$. We focus on the Edge-Triangle Model (a slight adjustment yields the 
same results for the Triangle Model). In Section~\ref{S4.2} we investigate how averages under 
the canonical ensemble, like \eqref{softconstraint2}, behave when $n\to\infty$. Here we can treat 
general constraints defined in \eqref{operator}.

For the behaviour of our constrained models, the \emph{sign} of the coordinates of the tuning 
parameter $\vec{\theta}^*$ is of pivotal importance, both for a fixed $n\in\mathbb{N}$ and 
asymptotically (see Bhamidi \emph{et al.}~\cite{BBS11}, Chatterjee and Diaconis~\cite{CD13}, 
Radin and Yin~\cite{RY13}, and references therein). We must therefore carefully keep track of 
this sign. The key results in this direction are Lemmas \ref{lem:tunealt} and~\ref{lem: fixed n sign} 
below.

%%%%

\subsection{Tuning parameter for fixed $n$}
\label{S4.1}
\begin{lemma}
\label{lem:tunealt}
Consider the Triangle Model with the constraint given by the triangle density $T_2^*$. For every $n$,
$\theta^* \geq 0$ if and only if $T_2^* \geq \tfrac{1}{8}$. 
\end{lemma}

\begin{proof}
The proof is similar to that of Lemma~\ref{lem: fixed n sign} below.  
\end{proof}

\begin{lemma}
\label{lem: fixed n sign}
Consider the Edge-Triangle Model. For every $n$, $\theta_2^*\geq0$ if and only if $T_2^*\geq\frac{1}{8}$, 
irrespective of $T_1^*$. Furthermore, $\theta_1^*\geq0$ if and only if $T_1^*\geq\frac{1}{2}$.  
\end{lemma}
\begin{proof}
Define, for $\theta_1,\theta_2\in\R$, the function 
\begin{equation}
g(\theta_1,\theta_2) := \sum_{G\in\cG_n} \exp\left[n^2\left(\theta_1(T_1(G)-\tfrac{1}{2})
+\theta_2(T_2(G) - \tfrac{1}{8})\right)\right].
\end{equation}
We first prove that $g$ attains a unique global minimum at $(\theta_1,\theta_2)=(0,0)$.
Consider the canonical ensemble $\Pcan$ as defined in \eqref{eq:CPD} and \eqref{canonicprob}, 
with $\vec{T}$ as defined above, and the probability distribution $\mathrm{P}_{\text{hom}}$ 
on $\cG_n$ that assigns probability $2^{-{n\choose2}}$ to every graph $G\in\cG_n$. Since 
$\mathrm{P}_{\text{hom}}$ is absolutely continuous with respect to $\Pcan$, the relative entropy 
$S_n(\mathrm{P}_{\text{hom}}|\Pcan) $ is well defined: 
\begin{equation} 
S_n(\mathrm{P}_{\text{hom}} \mid \Pcan) = \sum_{G\in\cG_n}
\mathrm{P}_{\text{hom}}(G)\log\frac{\mathrm{P}_{\text{hom}}(G)}{\Pcan(G)} \geq 0.
\end{equation}
Using the form of the canonical ensemble we get, after some straightforward calculations, that,
for all $\theta_1,\theta_2\in\R$,
\begin{equation}
\label{eq: REG} 
\sum_{G\in\cG_n} \exp\Big[n^2\big(\theta_1T_1(G) + \theta_2 T_2(G)\big)\Big] 
\geq 2^{{n\choose2}} \exp\left[n^2\big(\theta_1\tfrac{1}{2}+\theta_2\tfrac{1}{8}\big)\right],
\end{equation}
where the term in the right-hand side comes from the relation  
\begin{equation}
\sum_{G\in\cG_n} \frac{1}{2^{{n\choose2}}} \left(\theta_1T_1(G) +\theta_2T_2(G)\right) 
= \theta_1\tfrac{1}{2} + \theta_2\tfrac{1}{8}.
\end{equation}
Observe that the left-hand side represents the average edge and triangle density, multiplied with 
$\theta_1,\theta_2$, in an Erd\H{o}s-R\'enyi random graph with parameters $(n,\frac{1}{2})$.  
From (\ref{eq: REG}) we find that $g(\theta_1,\theta_2) \geq 2^{{n\choose2}} = g(0,0)$ for all 
$\theta_1,\theta_2\in\R$, and so $g$ attains a global minimum at $(0,0)$. In what follows we 
show that this global minimum is unique. A straightforward computation shows that $\partial_{\theta_1}
g(\theta_1,\theta_2) = \partial_{\theta_2} g(\theta_1,\theta_2) = 0$ if and only if $\langle T_1\rangle 
= \tfrac{1}{2}$ and $\langle T_2\rangle = \tfrac{1}{8}$. Furthermore, the Hessian matrix is a covariance 
matrix and hence is positive semi-definite. For $\vec{\theta} = (\theta_1,\theta_2) = (0,0)$ we know 
that $\langle T_1\rangle = \tfrac{1}{2}$ and $\langle T_2\rangle = \tfrac{1}{8}$. Hence, by uniqueness 
of the multiplier $\vec{\theta}^*$ for the constraint $T_1^* = \tfrac{1}{2}$, $T_2^* = \tfrac{1}{8}$, we 
obtain that $g$ has a unique global minimum at $(0,0)$. Moreover, this shows that $g$ has no other 
stationary points. Consider the parameter $(\theta_1,\theta_2) = (\theta_1^*,\theta_2^*)$. We have 
\begin{equation}
\begin{aligned}
\partial_{\theta_2}g(\theta_1^*,\theta_2^*) 
&= \left(\langle T_2 \rangle -\tfrac{1}{8}\right)\text{exp}[-n^2(\theta_1^*\tfrac{1}{2}+\theta_2^*\tfrac{1}{8})]
\sum_{G\in\cG_n}\exp\left[n^2\left(\theta_1^*T_1(G)+ \theta_2^*T_2(G)\right)\right] \\
&= \left(T_2^*-\tfrac{1}{8}\right)\text{exp}[-n^2(\theta_1^*\tfrac{1}{2}+\theta_2^*\tfrac{1}{8})]
\sum_{G\in\cG_n}\exp\left[n^2\left(\theta_1^*T_1(G)+ \theta_2^*T_2(G)\right)\right].
\end{aligned}
\end{equation}
If $T_2^*\geq\tfrac{1}{8}$, then $\partial_{\theta_2}g(\theta_1^*,\theta_2^*)\geq 0$. Because $g$ has a 
unique stationary point at $(0,0)$, which is a global minimum, we get $\theta_2^*\geq 0$. Similarly, 
we can show that if $T_2^*<\tfrac{1}{8}$, then $\theta_2^*<0$. Suppose that $T_1^* \geq \tfrac{1}{2}$. 
For the parameter $(\theta_1,\theta_2) = (\theta_1^*,\theta_2^*)$ we have 
\begin{equation}
\begin{aligned}
\partial_{\theta_1} g(\theta_1^*,\theta_2^*) 
&= \left(\langle T_1 \rangle -\tfrac{1}{2}\right)\text{exp}[-n^2(\theta_1^*\tfrac{1}{2}+\theta_2^*\tfrac{1}{8})]
\sum_{G\in\cG_n} \exp\left[n^2\left(\theta_1^*T_1(G)+ \theta_2^*T_2(G)\right)\right] \\
&= \left(T_1^*-\tfrac{1}{2}\right)\text{exp}[-n^2(\theta_1^*\tfrac{1}{2}+\theta_2^*\tfrac{1}{8})]
\sum_{G\in\cG_n}\exp\left[n^2\left(\theta_1^*T_1(G)+ \theta_2^*T_2(G)\right)\right].
\end{aligned}
\end{equation}
Arguing in a similar way as before, we conclude that $\theta_1^*\geq0$ if and only if $T_1^*\geq\frac{1}{2}$. 
\end{proof}

Consider the Edge-Triangle Model and suppose that the constraint $(T_1^*,T_2^*)$ is such that 
$T_2^* = T_1^{*3}$. Then $\theta_2^*=0$ and $\theta_1^*$ matches the constraint on the edge 
density only. The following lemma shows that in this case the canonical ensemble behaves like 
the Erd\H{o}s-R\'enyi model with parameter $T_1^*$, a fact that will be needed later to prove 
equivalence. 

\begin{lemma}
\label{lem:ER}
Consider the Edge-Triangle Model with the constraint given by the edge-triangle densities $\vec{T}^* 
= (T_1^*,T_2^*)$ with $T_2^* = T_1^{*3}$. Consider the canonical ensemble as defined in 
\eqref{canonicprob}. Then, for every $n\in\N$, 
\begin{equation}
\theta_1^* = \frac{1}{2}\log\frac{T_1^*}{1-T_1^*}, \qquad \theta^*_2=0.
\end{equation}

\end{lemma} 

\begin{proof}
From the definition of the canonical ensemble we have that, for $G\in\cG_n$,
\begin{equation}
\Pcan(G) = \Pcan(G \mid \vec{\theta}^*) 
= \eee^{n^2\left[\theta_1^*T_1(G) + \theta_2^*T_2(G) - \psi_n(\vec{\theta}^*)\right]},
\end{equation}
where $\psi_n(\vec{\theta}^*)$ is the partition function defined in \eqref{eq:PF}. For the 
specific value $\vec{\theta} = \vec{\theta}^*$ we have that (recall \eqref{softconstraint2})
\begin{equation}
\langle T_1\rangle = T_1^*, \qquad \langle T_2\rangle = T_2^* = T_1^{*3}.
\end{equation}
We claim that the correct parameter is $\vec{\theta}^*=(\frac{1}{2}\log\frac{T_1^*}{1-T_1^*},0)$. 
The average fraction of edges is $T_1^*$ (see Park and Newman~\cite{PN14}). The average 
number of triangles is
\begin{align*}
\langle T_2\rangle 
&= \frac{\sum_{G\in\cG_n}T_2(G)
\exp\left[n^2\left(\frac{1}{2}\log\frac{T_1^*}{1-T_1^*}\,T_1(G)\right)\right]}
{\sum_{G\in\cG_n}\exp\left[n^2\left(\frac{1}{2}\log\frac{T_1^*}{1-T_1^*}\,T_1(G)\right)\right]}\\
&= \frac{\sum_{G\in\cG_n}T_2(G)(T_1^*)^{E(G)}(1-T_1^*)^{{n\choose 2}-E(G)}}
{\sum_{G\in\cG_n}(T_1^*)^{E(G)}(1-T_1^*)^{{n\choose 2}-E(G)}} = T_1^{*3},
\end{align*}
where the last equation comes from the fact we are calculating the average number of triangles in 
an Erd\H{o}s-R\'enyi model with probability $T_1^*$. Since the multiplier $\vec{\theta}^*$ is unique, the 
proof is complete.
\end{proof}

%%%

\subsection{Tuning parameter for $n\to\infty$}
\label{S4.2}

In Lemma \ref{NMAX} below we show how averages under the canonical ensemble behave 
asymptotically when $\vec{\theta}$ does not depend on $n$. In Lemma \ref{lem:NMAX} we will
look at what happens when $\vec{\theta}$ is a one-dimensional multiplier and depends on $n$.

\begin{lemma}
\label{NMAX}
Suppose that the operator $\vec{T}\colon\,W \to \R^ m$ is bounded and continuous with 
respect to the $\delta_{\square}$-norm as defined in \eqref{deltam}. For $\vec{\theta}\in\R^m$ 
independent of $n$, consider the variational problem 
\begin{equation}
\label{eq:variational}
\sup_{\tilde{h}\in\tilde{W}} \big[\vec{\theta}\cdot\vec{T}(\tilde{h})-I(\tilde{h})\big],
\end{equation}
where $I$ is defined in \eqref{eq: rate function}. Suppose that the supremum is 
attained at a unique point, denoted by $\tilde{h}^*(\vec{\theta})$. Then
\begin{equation}
\label{eq:Maximizers}
\lim_{n\to\infty} \sum_{G\in\cG_n} T_k(G)\,\Pcan(G \mid \vec{\theta}\,) 
= T_k\big(\tilde{h}^*(\vec{\theta})\big), \qquad k=1,\ldots,m.
\end{equation}

\end{lemma}

\begin{proof}
The average of $T_k$ under the canonical probability distribution is equal to 
\begin{equation}
\sum_{G\in\cG_n} T_k(G)\Pcan(G \mid \vec{\theta}) 
= \sum_{G\in\cG_n}T_k(G)\,\eee^{n^2\left[\vec{\theta}\cdot\vec{T}(G)-\psi_n(\vec{\theta})\right]}
=: \mathcal{J}_n.
\end{equation}
Pick $\delta>0$ and consider the $\delta$-ball $B_{\delta}(\tilde{h}^*)$ around the maximiser 
$\tilde{h}^*$ in the quotient space $(\tilde{W},\delta_{\square})$, i.e., 
\begin{equation}
B_{\delta}(\tilde{h}^*) 
:= \big\{\tilde{h}\in\tilde{W}\colon \delta_{\square}(\tilde{h},\tilde{h}^*)<\delta\big\}.
\end{equation} 
We denote by $G^{\delta}$ a graph on $n$ vertices whose graphon is a representative element 
of the class $\tilde{h}^G$. With a slight abuse of notation, we denote by $G^{\delta}$ both the 
graph and the corresponding graphon, and by $\tilde{h}^G$ the corresponding equivalence class 
in the quotient space $(\tilde{W},\delta_{\square})$. Since $(\tilde{W},\delta_{\square})$ is compact 
space (recall Proposition~\ref{comp}), and the graphons associated with finite graphs form a 
countable family that is dense in $(\tilde{W}, \delta_{\square})$ (see Diao \emph{et al.}~\cite{DGKR15}, 
Lov\'asz and Szegedy~\cite{LS06}), there exists a sequence $(\tilde{h}^{G_n})_{n\in\N}$ such that 
$\lim_{n\to\infty} \delta_{\square}(\tilde{h}^{G_n}, \tilde{h}^*)=0$. For $n$ large enough the neighbourhood 
$B_{\delta}(\tilde{h}^*)$ contains elements of the sequence $(\tilde{h}^{G_n})_{n\in\N}$ and, due to the 
Lipschitz property (recall Proposition~\ref{Lip}), $\delta_{\square}(\tilde{h}^{G_n},\tilde{h}^*)<\delta$ 
implies $|T_k(\tilde{h}^{G_n})-T_k(\tilde{h}^*)|<C_k\delta$ for some constant $C_k>0$ and $k=1,\ldots,m$.
\paragraph{Upper bound for $\mathcal{J}_n$.} 
We decompose the sum over $G\in\cG_n$ into two parts: the first over $G$ whose graphon lies in 
$B_{\delta}(\tilde{h}^*)$, the second over $G$ whose graphon lies in $B_{\delta}(\tilde{h}^*)^c
=:\tilde{W}^{\delta,\#}$. We further denote by
\begin{equation}
\cG_n^{\delta} := \big\{G\in\cG_n\colon\, 
|T_k(\tilde{h}^G)-T_k(\tilde{h}^*)|<\delta, \hspace{2mm} k=1,\hdots,m\big\}, 
\end{equation}
the set of all graphs whose subgraph densities $T_k(G)$ are $\delta$-close to $T_k(\tilde{h}^*)$. A 
graph from this set is denoted by $G^{\delta}$. We define the set
\begin{equation}
\cG_n^{\delta,\#}
:=\big\{G\in\cG_n\colon\, \tilde{h}^G\in \tilde{W}^{\delta,\#}\big\}
\end{equation}
and, for $k=1,\ldots,m$, obtain the following upper bound:   
\begin{eqnarray}\nonumber
\mathcal{J}_n 
&= &\sum_{G\in\cG_n^\delta} T_k(G)\, \eee^{n^2\left[\vec{\theta}\cdot\vec{T}(G)-\psi_n(\vec{\theta})\right]} 
+\sum_{G\in\cG_n^{\delta,\#}} T_k(G)\,\eee^{n^2\left[\vec{\theta}\cdot\vec{T}(G)-\psi_n(\vec{\theta})\right]}\\
\nonumber
&\leq &\frac{(T_k(G^{\delta})+\delta) \sum\limits_{G\in\cG_n^\delta} \eee^{n^2\vec{\theta}\cdot\vec{T}(G)}}
{\sum\limits_{G\in\cG_n^\delta} \eee^{n^2\vec{\theta}\cdot\vec{T}(G)}} 
+ \sum_{G\in\cG_n^{\delta,\#}} T_k(G)\,\eee^{n^2\left[\vec{\theta}\cdot\vec{T}(G)-\psi_n(\vec{\theta)}\right]}\\
\label{eq: sum}&=& (T_k(G^{\delta})+\delta) 
+ \frac{\sum\limits_{G\in\cG_n^{\delta,\#}}T_k(G)\,\eee^{n^2\vec{\theta}\cdot\vec{T}(G)}}
{\sum\limits_{G\in\cG_n} \eee^{n^2 \vec{\theta}\cdot\vec{T}(G)}}.
\end{eqnarray}
Next, we further bound the second term in \eqref{eq: sum}. By definition, for every $n\in \N$ the
range of the operator $\vec{T}$ is a finite set
\begin{equation}
R_n := \big\{\vec{g} \in [0,\infty)^m\colon\,\vec{T}(G) = \vec{g}, ~G\in\cG_n\big\}.
\end{equation} 
For the set $R_n$ we observe that $|R_n| = o(n^{m^2})$. In addition, introduce the sets 
\begin{equation}
\begin{aligned}
\cG_n^{\vec{g}}
&:=\{G\in\cG_n\colon\, \vec{T}(G) = \vec{g}\},\\
R_n^{\delta,\#}
&:=\{\vec{g} \in [0,\infty)^m\colon\, \vec{T}(G) = \vec{g}, G\in\cG_n^{\delta,\#}\}\subset R_n.
\end{aligned}
\end{equation} 
The operator $\vec{T}$ is bounded, and so there exists an $M>0$ such that  $\|\vec{T}(G)\| 
\leq M$ for all $G\in\cG_n$. Hence, the second term in (\ref{eq: sum}) can be bounded from
above by
\begin{equation}
\label{eq:M}
\frac{\sum_{G\in\cG_n^{\delta,\#}} T_k(G)\,\eee^{n^2\vec{\theta}\cdot\vec{T}(G)}}
{\sum_{G\in\cG_n} \eee^{n^2 \vec{\theta}\cdot\vec{T}(G)}} 
\leq \frac{|R_n^{\delta,\#}|\,M\,\exp\big[n^2\sup_{\vec{g}\in R_n^{\delta,\#}}
(\vec{\theta}\cdot\vec{g} + \frac{1}{n^2}\log|\cG_n^{\vec{g}}|)\big]}
{\exp\big[n^2\sup_{\vec{g}\in R_n}(\vec{\theta}\cdot\vec{g} 
+ \frac{1}{n^2}\log|\cG_n^{\vec{g}}|)\big]}.
\end{equation}
By the large deviation principle in Theorem~\ref{th:LDP}, we have
\begin{equation}
\frac{1}{n^2}\,\log|\cG_n^{\vec{g}}| = \inf_{\tilde{h}\in\tilde{W}^{\vec{g}}}I(h) + o(1),
\end{equation} 
where $\tilde{W}^{g} = \{\tilde{h}\in\tilde{W}\colon\,\vec{T}(\tilde{h}) = \vec{g}\}$. As a consequence, 
\eqref{eq:M} is majorised by
\begin{equation}
\begin{aligned}
\lefteqn{M\,|R_n^*|\,\eee^{o(n^2)}\exp\left[n^2\left(\,\sup_{\vec{g}\in R_n^{\delta,\#}}
\big[\vec{\theta}\cdot\vec{g} - \inf_{\tilde{h}\in\tilde{W}^{\vec{g}}} I(\tilde{h})\big]
- \sup_{\vec{g}\in R_n} \big[\vec{\theta}\cdot\vec{g} - \inf_{\tilde{h}\in\tilde{W}^{\vec{g}}} 
I(\tilde{h})\big]\right)\right]}\\
&=
M\,|R_n^*|\, \eee^{o(n^2)}\exp\left[n^2\left(\,\sup_{\vec{g}\in R_n^{\delta,\#}}
\sup_{\tilde{h}\in \tilde{W}^{\vec{g}}}\big[\vec{\theta}\cdot\vec{T}(\tilde{h}) -  I(\tilde{h})\big]
- \sup_{\vec{g}\in R_n} \sup_{\tilde{h} \in \tilde{W}^{\vec{g}}} \big[\vec{\theta}\cdot\vec{T}(\tilde{h}) 
-  I(\tilde{h})\big]\right)\right]\\
&=
M\,|R_n^*| \,\eee^{o(n^2)}\exp\left[n^2\left(\,\sup_{\tilde{h}\in \tilde{W}^{\delta,\#}}
\big[\vec{\theta}\cdot\vec{T}(\tilde{h}) -  I(\tilde{h})\big]
- \sup_{\tilde{h}\in \tilde{W}}\big[\vec{\theta}\cdot\vec{T}(\tilde{h}) 
-  I(\tilde{h})\big]\right)\right]. 
\label{eq:approx}
\end{aligned}
\end{equation}
The last equation can be justified as follows. Define the sets
\begin{equation}
\label{eq:GGC}
\tilde{W}_n = \big\{\tilde{h}\in\tilde{W}\colon\,\tilde{h}=\tilde{h}^G \text{ for some graph } 
G\in\cG_n\big\}, \qquad \tilde{W}^{\delta,\#}_n=\tilde{W}^{\delta,\#}\cap\tilde{W}_n.
\end{equation}
Since the graphons associated with finite graphs form a countable set that is dense in 
$(\tilde{W},\delta_{\square})$, we have that 
\begin{equation}
\label{eq:Dense}
\tilde{W} = \text{cl}\left(\bigcup\limits_{n\in\N} \tilde{W}_n\right), \qquad 
\tilde{W}^{\delta,\#} = \text{cl}\left(\bigcup\limits_{n\in\N} \tilde{W}^{\delta,\#}_n\right),
\end{equation}
where $\text{cl}$ denotes closure. Using \eqref{eq:Dense}, and recalling that $\vec{T}$ is continuous 
and $I$ is lower-semicontinuous, we get
\begin{equation}
\limn \sup_{\vec{g}\in R_n^{\delta,\#}} \sup_{\tilde{h}\in \tilde{W}^{\vec{g}}}
\big[\vec{\theta}\cdot\vec{T}(\tilde{h}) -  I(h)\big] 
= \sup_{\tilde{h}\in\tilde{W}^{\delta,\#}}\big[\vec{\theta}\cdot\vec{T}(\tilde{h})-I(\tilde{h})\big],
\end{equation}
and a similar result can be established for the second supremum in the exponent in \eqref{eq:approx}. 
The exponent in \eqref{eq:approx} is negative for all $\delta>0$ and is independent of $n$. Moreover, 
by the left-continuity of the graph sequence $(G_n^{\delta})_{n\in\N}$, we have that $\limn T_k(G_n^{\delta}) 
= T_k(\tilde{h}^*)$ for every $k=1,\ldots,m$ and every $\delta>0$. Combined with the inequality in 
\eqref{eq: sum}, we obtain, for $k=1,\ldots,m$,
\begin{equation}
\limn \sum_{G\in\cG_n} T_k(G)\, 
\eee^{n^2\left[\vec{\theta}\cdot\vec{T}(G)-\psi_n(\vec{\theta})\right]} \leq T_k(\tilde{h}^*).
\end{equation}

\paragraph{Lower bound for $\mathcal{J}_n$.} 
We distinguish two cases: $T_k(\tilde{h}^*) = 0$ and $T_k(\tilde{h}^*)>0$. For the first case we 
trivially get the lower bound 
\begin{equation} 
\limn \sum_{G\in\cG_n} T_k(G)\, \eee^{n^2\vec{\theta}\cdot\vec{T}(G)}\geq 0 = T_k(\tilde{h}^*).
\end{equation} 
For the second case we show the equivalent upper bound for the inverse, i.e.,
\begin{equation}
\limn \frac{\sum_{G\in\cG_n} \eee^{n^2\vec{\theta}\cdot\vec{T}(G)}}
{\sum_{G\in\cG_n} T_k(G)\, \eee^{n^2\vec{\theta}\cdot\vec{T}(G)}} \leq \frac{1}{T_k(\tilde{h}^*)}.
\end{equation}
Using the fact that $T_k(\tilde{h}^*) \neq 0$ is bounded, and using a similar reasoning as for the 
upper bound on $\mathcal{J}_n$, the latter is easily verified.
\end{proof}

\begin{remark}
The convergence in \eqref{eq:Maximizers} is not necessarily uniform in $\vt$. Our results in Theorem \eqref{thm:equivalence} (II)(b)-(II)(d) indicate that breaking of ensemble equivalence manifests itself through non uniform convergence in \eqref{eq:Maximizers}. In Lemma \eqref{lem:NMAX} we show that uniform convergence holds when the constraint is on the triangle density only, which explains our result in Theorem \eqref{thm:equivalence} (I).
\end{remark}

\begin{remark}
{\rm The analogue of Lemma~\ref{NMAX} when the supremum in \eqref{eq:variational} has 
multiple maximisers in $\tilde{W}$ is considerably more involved.} \hfill \qed
\end{remark}

As observed in Remark~\ref{rem:ndep}, in general the tuning parameter $\vec{\theta}^*$ depends 
on $n$. We discuss this dependence in Appendix~\ref{app}.

%%%%%%% SECTION 5 %%%%%%%%%%%%%%

\section{Proof of Theorem~\ref{thm:equivalence}}
\label{S5}

We proceed by computing the relative entropy $s_{\infty}$. In Sections \ref{S5.1}, \ref{S5.4}, 
\ref{S5.5}, \ref{S5.6}, \ref{S5.7} and \ref{S5.9} we treat the limiting regime where all constraints 
and parameters are the limiting parameters as in \eqref{eq: Assumption T} and \eqref{eq:Assumption}. 
In Sections \ref{S5.2} and \ref{S5.8} we write $T^*_{\infty,1}, T^*_{\infty,2},\theta^*_{\infty,1}$ for the 
limiting regime.

%%%

\subsection{Proof of (I)(a)~(Triangle model ~$T_2^*\geq\frac{1}{8})$}
\label{S5.1}

\begin{proof}
Theorem~\ref{th:Limit} says that
\begin{equation}
\label{eq:srep1}
s_{\infty} =  \sup_{\tilde{h}\in \tilde{W}} \big[\theta^*T_2(\tilde{h})-I(\tilde{h})\big]
-\sup_{\tilde{h}\in \tilde{W}^*} \big[\theta^*T_2(\tilde{h}) - I(\tilde{h})\big].
\end{equation}
Consider the first term in the right-hand side \eqref{eq:srep1}. From Lemma~\ref{lem:tunealt} 
we know that $\theta^*\geq 0$ if and only if $T_2^*\geq\tfrac{1}{8}$. From Theorem 
\ref{thm: CD variational simplify} it follows that if $\theta^* \geq 0$, then
\begin{equation}
\sup_{\tilde{h}\in \tilde{W}} \big[\theta^*T_2(\tilde{h})-I(\tilde{h})\big] 
= \sup_{u \in [0,1]} \big[\theta^* u^3 - I(u)\big] 
= \sup_{u \in [0,1]} \ell_3(u;\theta^*).
\end{equation}
From Radin and Yin~\cite[Proposition 3.2]{RY13} we know that $\ell_3(u,\theta^*)$ attains a 
unique global maximum. Let $u^*(\theta^*) = \text{arg} \sup_{u \in [0,1]} \ell_3(u;\theta^*)$ 
be the unique global maximiser. Using Lemma~\ref{lem:NMAX}, we obtain that $u^*(\theta^*) 
= {T_2^*}^{1/3}$, which leads to 
\begin{equation}
\label{e:CANTalt}
\sup_{u \in [0,1]} \ell_3(u;\theta^*) 
= \theta^* u^*(\theta^*)^{3} - I\big(u^*(\theta^*)\big) 
= \theta^* T_2^* - I\big(T_2^{*1/3}\big).
\end{equation}
As to the second term in the right-hand side of \eqref{eq:srep1}, we use Chatterjee and 
Varadhan~\cite[Proposition 4.2]{CV11}, which states that, for $T_2^*\in (\tfrac18,1]$,
\begin{equation}
\inf_{\tilde{h}\in\tilde{W}}I(\tilde{h})
:=\inf\big\{I(\tilde{h})\colon\,\tilde{h}\in\tilde{W}, T_2(\tilde{h})= T_2^*\big\} 
= \inf\big\{I(\tilde{h})\colon\,\tilde{h}\in\tilde{W}, T_2(\tilde{h})\geq T_2^*\big\}.
\end{equation}
Moreover, $I$ is convex at the point $x=T_2^{*1/3}$, and hence from Chatterjee and 
Varadhan~\cite[Theorem 4.3]{CV11} we have that $\inf_{\tilde{h}\in\tilde{W}^*}I(\tilde{h}) 
= I(T_2^{*1/3})$. Combining this with \eqref{e:CANTalt}, we conclude that $s_{\infty}=0$. 
\end{proof}

%%%

\subsection{Proof of (I)(b)~($T_2^*=0$)}
\label{S5.2}

Consider the Triangle Model with the constraint given by the triangle density $T^*=0$. It was 
proven by Erd\H{o}s {\it et al.}~\cite{EKR73} that almost all triangle-free graph have a bipartite 
structure. For the case of dense graphs, the condition $T^*=0$ means that the number of 
triangles in the graph is of order $o(n^2)$. In the proof we will see that the two ensembles 
are equivalent and that graphs drawn from the two ensembles have a bipartite structure. 

\begin{proof}
From the construction of the canonical ensemble $\Pcan$ in Section \ref{S1.2}, we observe that 
$\Pcan(G) = 0$ when $T(G)>0$. This is a direct consequence of (\ref{softconstr}). We write
\begin{equation}
\label{eq: Gno}
\cG_n^0 :=\{G\in\cG_n: T(G)=0\}
\end{equation}
for the collection of all graphs with triangle density equal to zero. From (\ref{eq:PC}) we obtain that 
$\Pcan(G) = 0$ if $G\notin \cG_n^0$ and $\Pcan(G) = |\cG_n^0|^{-1}$ if $G\in\cG_n^0$. Hence
$\Pcan(G) = \Pmic(G)$ when the constraint is given by $T^*=0$, which yields 
\begin{equation}
S_n(\Pmic \mid \Pcan) = 0 \qquad  \forall n\in\N
\end{equation}
and hence $s_\infty = 0$.
\end{proof}

%%%

\subsection{Proof of (II)(a)~(Edge-Triangle model ~$T_2^*=T_1^{*3}$)}
\label{S5.4}
 
For the case $T_1^* = T_2^{*\frac{1}{3}}$ we have shown in Lemma~\ref{lem:ER} that the canonical 
ensemble essentially behaves like an Erd\H{o}s-R\'enyi model with parameter $p= T_1^*$. Furthermore, 
the microcanonical ensemble also has an explicit expression, which is found by using the following lemma. 

\begin{lemma}
\label{le:MicET}
If $T_1^* = T_2^{*\frac{1}{3}}$, then 
\begin{equation}
\inf_{\tilde{h}\in\tilde{W}^*} I(\tilde{h}) = I\big(T_2^{*\frac{1}{3}}\big)=I\big(T_1^*\big).
\end{equation}
\end{lemma}

\begin{proof}
Consider an element $\tilde{h}\in\tilde{W}^*$ with $\tilde{W}^*:=\{\tilde{h}\in\tilde{W}\colon\,T_1(\tilde{h}) 
= T_1^* = T_2^{*\frac{1}{3}}, T_2(\tilde{h}) = T_2^*\}$. Using the convexity of $I$ on $\tilde{W}$ and 
Jensen's inequality, we get 
\begin{equation}
I(\tilde{h}) = \int_{[0,1]^2} \dd x\,\dd y\, I(h(x,y)) \geq I\left(\int_{[0,1]^2} \dd x\,\dd y\, h(x,y)\right) 
= I\big(T_1(\tilde{h})\big)= I(T_1^*) = I\big(T_2^{*\frac{1}{3}}\big).
\end{equation}
Hence $I(\tilde{h})\geq I(T_2^{*\frac{1}{3}})$ for every $\tilde{h}\in\tilde{W}^*$, which proves the claim.  
\end{proof}

\begin{proof}[Proof of (II)(a)]
Consider the relative entropy $s_{\infty}$ as defined in (\ref{eq:Equivalence}) and (\ref{eq:KL2}). Using 
Lemma~\ref{lem:ER}, we obtain the expression 
\begin{equation}
s_{\infty} = -\frac{1}{2}T_1^*\log(T_1^*)-\frac{1}{2}(1-T_1^*)\log(1-T_1^*)+\inf_{\tilde{h}\in\tilde{W}^*}I(\tilde{h}).
\end{equation}
From Lemma \ref{le:MicET} we have that $\inf\limits_{\tilde{h}\in\tilde{W}^*}I(\tilde{h}) = I(T_1^*)$, 
which yields $s_{\infty}=0$. 
\end{proof}

%%%

\subsection{Proof of (II)(b)~($T_2^*\neq T_1^{*3}$ and $T_2^*\geq\frac{1}{8}$)}
\label{S5.5}
 
\begin{proof}
From Lemma~\ref{lem:  fixed n sign} we know that if $T_1^*\geq \frac{1}{2}$ and $T_2^*\geq\frac{1}{8}$,
then $\theta_{1}^*\geq0$ and $\theta_{2}^*\geq0$ while if $T_1^*< \frac{1}{2}$ and $T_2^*\geq
\frac{1}{8}$, then $\theta_{1}^*<0$ and $\theta_{2}^*\geq0$. An argument similar as above yields 
\begin{equation}
\label{eq: RY}
\sup\limits_{\tilde{h}\in\tilde{W}}\left[\theta_{1}^*T_1(\tilde{h})
+\theta_{2}^*T_2(\tilde{h})-I(\tilde{h})\right] 
= \sup_{u \in [0,1]} \ell_3(u;\vec{\theta}^*),
\end{equation}
where for $\theta_{1}^*\geq0$ and $\theta_{2}^*\geq0$ the last supremum has a unique solution 
(see Radin and Yin~\cite[Proposition 3.2]{RY13}), while for $\theta_{1}^*<0$ and $\theta_{2}^*
\geq0$ it either has a unique solution or two solutions. We treat these two cases separately.

\paragraph{Unique solution.} 
Because of the uniqueness of the solution, not all realisable hard constraints can be met in the limit 
(see Lemma~\ref{NMAX}). We observe that,  if $T_2^*\geq\frac{1}{8}$ and $T_2^*\neq T_1^{*{3}}$,
in the limit as $n\to\infty$ the canonical ensemble becomes Erd\H{o}s-R\'enyi with parameter $p$. This 
regime is known as the \emph{high-temperature} regime (see Bhamidi \textit{ et al.} \cite{BBS11} and 
Chatterjee and Diaconis \cite{CD13}). In what follows we determine the parameter $p$ of the canonical 
ensemble in the limit. From Bhamidi \emph{et al.}~\cite[Theorem 7]{BBS11} we have that 
$p = u^*(\vec{\theta}^*)^{\frac{1}{3}}$ with $ u^*(\vec{\theta}^*)^{\frac{1}{3}}$ the unique 
maximiser of (\ref{eq: RY}). The expression in (\ref{eq: RY}) thus takes the form
\begin{equation}
\begin{aligned}
&\sup\limits_{\tilde{h}\in\tilde{W}}\left[\theta_{1}^*T_1(\tilde{h})
+\theta_{2}^*T_2(\tilde{h})-I(\tilde{h})\right]\\ 
&= \sup_{u \in [0,1]} \ell_3(u;\vec{\theta}^*) 
= \theta_{1}^* u^*(\vec{\theta}^*)^{\frac{1}{3}} 
+ \theta_{2}^*u^*(\vec{\theta}^*) 
- I\big(u^*(\vec{\theta}^*)^{\frac{1}{3}}\big).
\end{aligned}
\end{equation}
Consider the second term in the right-hand side of \eqref{varreprsinfty}. From the definition of 
$\tilde{W}^*$ it is straightforward to see that
\begin{equation}
\sup_{\tilde{h}\in\tilde{W}^*}\left[\theta_{1}^*T_1(\tilde{h}) 
+ \theta_{2}^*T_2(\tilde{h})-I(\tilde{h})\right] 
= \theta_{1}^* T_1^* +\theta_{2}^* T_2^* -\inf_{\tilde{h}\in\tilde{W}^*}I(\tilde{h}),
\end{equation}
where $\tilde{W}^* = \{ \tilde{h}\in\tilde{W}\colon\,T_1(\tilde{h}) = T_1^*, T_2(\tilde{h}) = T_2^*\}$. 
We observe that, due to $T_2^*\neq T_1^{*{3}}$, the constant function $h \equiv u^*
(\vec{\theta}^*)^{\frac{1}{3}}$ does not lie in $\tilde{W}^*$. This shows that $s_{\infty}>0$. 

\paragraph{Two solutions.} 
The regime in which the right-hand side of (\ref{eq: RY}) has two solutions is known as the 
\emph{low-temperature} regime. In this case the hard constraints $(T_1^*,T_2^*)$, with 
$T_1^*\in[\frac{1}{4},\frac{1}{2})$, $T_2^*\geq\frac{1}{8}$, lie on a curve on the $(T_1, T_2)$-plane 
in such a way such that the tuning parameters $(\theta_{1}^*,\theta_{2}^*)$ lie on the \emph{phase 
transition} curve found in Chatterjee and Diaconis~\cite{CD13} and Radin and Yin~\cite{RY13}. 
Denote the two solutions of (\ref{eq: RY}) by $u_1^*,u_2^*$. Because of the constraint we are 
considering, we have that neither of them lies in $\tilde{W}^*$. From the compactness of the 
latter space we see that $s_{\infty}>0$.
\end{proof}

%%%

\subsection{Proof of (II)(c)~($T_2^*\neq T_1^{*3}$, $0<T_1^*\leq\frac{1}{2}$ and $0<T_2^{*3}<\frac{1}{8}$)}
\label{S5.6}

For the case $0<T_1^*\leq \tfrac{1}{2}$, $T_2^*<\tfrac{1}{8}$ we know from Lemma~\ref{lem: fixed n sign} 
that $\theta^*_1 \leq 0$ and $\theta^*_2<0$ for every $n$. Hence, because of \eqref{eq:Assumption}, 
we have that $\theta_{1}^*\leq 0$ and $\theta_{2}^*<0$. This regime is significantly harder to 
analyse than the previous regimes. Consider the relative entropy $s_{\infty}$ and the variational 
representation given in \eqref{varreprsinfty}. We consider two cases: $T_2^*>T_1^{*3}$ and 
$T_2^*\leq T_1^{*3}$. 

\paragraph{Case $T_2^*>T_1^{*3}$.}
In this case we have the straightforward inequality 
\begin{equation}
s_{\infty}\geq \theta_{2}^*\left(T_1^{*3} - T_2^*\right) - I(T_1^*) 
+\inf_{\tilde{h}\in\tilde{W}^*}I(\tilde{h}).
\end{equation}
Since $T_1^{*3}<T_2^*$, we have $\theta_{2}^*\left(T_1^{*3} - T_2^*\right)>0$. We show that 
\begin{equation}
\label{eq:III}
\inf_{\tilde{h}\in\tilde{W}^*} I(\tilde{h}) =\inf\{I(\tilde{h})\colon\,\tilde{h}\in\tilde{W},\,T_1(\tilde{h})
=T_1^*,\,T_2(\tilde{h})=T_2^*\} > I(T_1^*).
\end{equation}
Using the convexity of $I$ on $\tilde{W}$ and Jensen's inequality, we obtain that $I(\tilde{h}) \geq 
I(T_1^*)$ for all $\tilde{h}\in\tilde{W}^*$. Hence 
\begin{equation}
\inf_{\tilde{h}\in\tilde{W}^*} I(\tilde{h})>\inf\{I(\tilde{h})\colon\,\tilde{h}\in\tilde{W},\,T_1(\tilde{h})=T_1^*\} 
= I(T_1^*),
\end{equation}
which settles (\ref{eq:III}). Hence $s_{\infty}>0$. 

\paragraph{Case $T_2^*\leq T_1^{*3}$.} 
We argue similarly as above. We have the straightforward inequality 
\begin{equation}
s_{\infty}\geq \theta_{1}^*\left(T_2^{*\frac{1}{3}} - T_1^*\right)
-I(T_2^{*\frac{1}{3}})+\inf_{\tilde{h}\in\tilde{W}}I(\tilde{h}).
\end{equation}
We have seen above that $\inf_{\tilde{h}\in\tilde{W}}I(\tilde{h})>I(T_1^*)$. We further now that $I$ 
is decreasing on $[0,\frac{1}{2}]$, and so $I(T_1^*)\geq I(T_2^{*{1}/{3}})$. Hence $s_{\infty}>0$. 

%%%

\subsection{Proof of (II)(d)~($(T_1^*,T_2^*)$ on the scallopy curve)}
\label{S5.7}

We show that if $(T_1^*, T_2^*)$ lies on the lower blue curve in Fig.~\ref{fig-scallopy} 
(referred to as the scallopy curve), then $s_{\infty}>0$. The case where $T_2^*\geq\frac{1}{8}$ 
can be dealt with directly via Theorem \label{thm: equivalence}(II)(b). The proof below deals with 
the case $T_2^*<\frac{1}{8}$. 

\begin{proof}
We give the proof for $\ell =2$, the extension to $\ell>2$ being similar. 

Suppose that $T_1^* = \tfrac12 + \epsilon$ with $\epsilon\in(0,\tfrac{1}{6})$, and that $T_2^*$ is chosen 
as small as possible. It is known that graphs with a relatively high edge density and with a triangle density 
that is as small as possible have a $d$-partite structure with edges added in a suitable way so 
that the desired triangle density is obtained (see Radin and Sadun~\cite{RS15} and Pikhurko and 
Raborov~\cite{PR12}). Consider a graph on $n$ vertices, denoted by $G$, with edge density 
$T_1\in(\tfrac{1}{2}, \tfrac{2}{3})$ and triangle density as small as possible. The structure 
of such graphs has been described above before Section~\ref{S4}. The graphon counterpart of 
such graphs is the optimiser of the second supremum in the right-hand side of the variational 
formula for $s_\infty$. Using Radin and Sadun~\cite[Theorem 4.2]{RS15}, we obtain  
\begin{equation}  
\sup_{\tilde{h}\in \tilde{W}^*} \big[\theta_{1}^*T_1(\tilde{h})
+\theta_{2}^*T_2(\tilde{h}) - I(\tilde{h})\big] 
= \theta_{1}^*T_1^* +\theta_{2}^* T_2^* - \frac{(1-c(\epsilon))^2}{2}\, I(p(\epsilon)),
\end{equation}
where 
\begin{equation}
c(\epsilon) = \frac{2+\sqrt{1-6\epsilon}}{6}, \qquad 
p(\epsilon) = \frac{4c(\epsilon)(1-2c(\epsilon))}{(1-c(\epsilon))^2}.
\end{equation}
In order to lighten the notation, we drop the dependence of $c$ and $p$ on $\epsilon$. Furthermore, 
the optimising graphon has the form
\begin{equation}
h^*_{\epsilon}(x,y) = \left\{ \begin{array}{ll}
1 &\mbox{if $x<c<y$ or $y<c<x$},\\
p &\mbox{if $c<x<\frac{1+c}{2}<y$ or $c<y<\frac{1+c}{2}<x$},\\
0 &\mbox{otherwise},
\end{array} \right.
\qquad (x,y)\in[0,1]^2,
\end{equation}
which has triangle density
\begin{equation}
T_2( h_{\epsilon}^*) = \frac{(2+\sqrt{1-6\epsilon})^2}{36}\,\frac{1-\sqrt{1-6\epsilon}}{3} = T(\epsilon).
\end{equation}
Let $\tilde{\mathcal{F}}_{\epsilon}$ be the set of all maximisers of $\theta_{1}^*\hspace{1pt}T_1
(\tilde{h})+\theta_{2}^*\hspace{1pt}T_2(\tilde{h}) - I(\tilde{h})$ on $\tilde{W}$. We show that 
$h_{\epsilon}^*\notin\tilde{\mathcal{F}}_{\epsilon}$, which yields $s_{\infty}>0$. From Chatterjee 
and Diaconis~\cite[Theorem 6.1]{CD13} we know that if $\tilde{h}\in\tilde{W}$ maximises 
$\theta_{1}^*\hspace{1pt}T_1(\tilde{h})+\theta_{2}^*\hspace{1pt}T_2(\tilde{h}) - I(\tilde{h})$ on 
$\tilde{W}$, then it must satisfy the Euler-Lagrange equations and it must be bounded away from 
0 and 1. Hence we see that $\tilde{h}_{\epsilon}^*$ cannot be a stationary point of $\theta_{1}^*
\hspace{1pt}T_1(\tilde{h})+\theta_{2}^*\hspace{1pt}T_2(\tilde{h}) - I(\tilde{h})$ on $\tilde{W}$,
and hence cannot be a maximiser. 
\end{proof}

%%%

\subsection{Proof of (II)(e)~($0<T_1^*\leq\frac{1}{2}$ and $T_2^*=0$)}
\label{S5.8}

\begin{proof} 
Consider the Edge-Triangle Model with constraint given by the edge and triangle densities 
$T_1^*\in(0,\frac{1}{2}]$ and $T_2^*=0$. Working as in Section~\ref{S5.2}, we find that the 
canonical ensemble assigns positive probability only to graphs satisfying the constraint 
$T_2^*=0$. Defining $\mathcal{G}_n^{0}$ as in (\ref{eq: Gno}) we obtain
\begin{equation}
\label{eq:CanNull}
\Pcan(G \mid \vec{\theta}) = \left\{ \begin{array}{ll}
e^{n^2\left[\theta_1 T_1(G) -\psi_n(\vec{\theta})\right]} &\mbox{if $G\in\cG_n^0$},\\
0 &\mbox{else},
\end{array} \right. 
\end{equation}
where $\psi_n(\vec{\theta}) = \sum_{G\in\cG_n^0}e^{n^2\theta_1 T_1(G)}$ is the partition function. 
From \eqref{eq:CanNull} we observe that the canonical probability distribution depends only on 
the edge parameter $\theta_1$. 
The parameter $\theta_1$ is chosen equal to $\theta^*_1$ that matches the soft constraint, i.e., 
\begin{equation}
\sum_{G\in\cG_n^0} T_1(G)\,\Pcan(G~|~\vec{\theta}^*) = T_1^*.
\end{equation} 
Arguing as in the proof of Chatterjee and Diaconis \cite[Theorem 3.1]{CD13} we find that the 
relative entropy equals 
\begin{equation}
s_{\infty} = \sup_{\tilde{h}\in\tilde{W}^0}\big[\theta^*_{\infty,1} T_1(\tilde{h})-I(\tilde{h})\big] 
- \sup_{\tilde{h}\in\tilde{W}^*}\big[\theta^*_{\infty,1} T_1(\tilde{h})-I(\tilde{h})\big],
\end{equation}
where 
\begin{equation}
\tilde{W}^0 := \{\tilde{h}\in\tilde{W}\colon\,T_2(\tilde{h})=0\}, 
\qquad \tilde{W}^*:=\{\tilde{h}\in\tilde{W}\colon\,T_1(\tilde{h})=T_{\infty,1}^*~, T_2(\tilde{h})=0\}.
\end{equation}
Using Chatterjee and Diaconis~\cite[Theorem 7.1 and Theorem 8.2]{CD13}, we obtain that 
$s_{\infty}=0$. 
\end{proof}

%%%

\subsection{Proof of (III)~(Star model ~ $T[j]^*\geq0$)}
\label{S5.9}

\begin{proof}
From Chatterjee and Diaconis~\cite[Theorem 6.4]{CD13} we have that, for all $\theta^*_{\infty}\in\R$, 
\begin{equation}
\sup_{\tilde{h}\in\tilde{W}} \big[ \theta^* W(\tilde{h})-I(\tilde{h})\big] 
= \sup_{u \in [0,1]} \big[\theta^* u^2-I(u)\big],
\end{equation}
which by Radin and Yin~\cite[Proposition 3.1]{RY13} has a unique solution, which we denote by 
$u^*(\theta^*)$. Using Theorem~\ref{th:Limit} we get that 
\begin{equation}
s_{\infty} = \theta^* u^*(\theta^*)^2  - I(u^*(\theta^*)) 
- \theta^*T^*+\inf_{\tilde{h}\in\tilde{W}^*}I(\tilde{h}),
\end{equation}
where, by Lemma~\ref{lem:NMAX}, we have that $u^*(\theta^*) 
= T^{*\frac{1}{2}}$. This yields
\begin{equation}
s_{\infty} = -I\big(T^{*\frac{1}{2}}\big) + \inf_{\tilde{h}\in\tilde{W}^*}I(\tilde{h}).
\end{equation}
We show that $\inf_{\tilde{h}\in\tilde{W}^*}I(\tilde{h}) = I(T^{*\frac{1}{2}})$. This is done by slightly 
modifying the proof of Chatterjee and Diaconis~\cite[Theorem 6.4]{CD13}. Indeed, observe that 
\begin{equation}
T[j](h) = \int_{[0,1]} \dd x\, M(x)^j, \qquad M(x) = \int_{[0,1]} \dd y\, h(x,y).
\end{equation}
Since $I$ is convex we have 
\begin{equation}
\label{eq:starineq}
\int_{[0,1]^2} \dd x\, \dd y\,I(h(x,y)) \geq \int_{[0,1]} \dd x\, I(M(x)), \qquad h \in W,
\end{equation}
with equality if and only if $h(x,y)$ is the same for almost all $y$. Since $h$ is a symmetric 
function, we get that equality holds if and only if $h$ is constant. For the constant function 
$h \equiv (T_j)^{1/j} \in W^* :=\{h\in W: T_j(h) = T_j\}$, \eqref{eq:starineq} is an equality. 
Hence, for any minimiser of $I$ on $\tilde{W}^*$ the inequality must be an equality, and 
thus any minimiser must be constant. This shows that $s_{\infty}=0$. 
\end{proof}

%%%%%%%%%%% APPENDIX %%%%%%%%%%%%%%%%%%%%%%%%%

\appendix 

\section{Appendix}
\label{app}

In this appendix we elaborate on the assumption made in \eqref{eq:Assumption}, i.e., the 
multiplier $\vec{\theta}^*_n$ converges to a limit $\vec{\theta}^*_{\infty}$ as $n\to\infty$.
In order to get a meaningful limit, we consider constraints $\vec{T}^*_n$ such that
\begin{equation}
\lim_{n\to\infty} \vec{T}^*_n = \vec{T}^*_{\infty}.
\end{equation}
It is straightforward to deduce from Corollary \ref{Cor: Count} and \eqref{eq:HF}--\eqref{softconstraint2} 
that if $\{\vec{T}_n^*\}$ is bounded away from 0 and 1 component-wise, then $(\vec{\theta}^*_n)_{n\in\N}$ 
is bounded away from $-\infty$ and $+\infty$ component-wise. Such a sequence contains a converging 
subsequence, say, $(\vec{\theta}^*_{n_k})_{k\in\N}$, which in general need not be unique. Thus, as long 
as the constraint is component-wise bounded away from 0 and 1, the asymptotic expressions derived in 
this paper exist, but their values may depend on the subsequence we choose. The value of $s_{\infty}$ 
depends on the chosen subsequence, but whether it is positive or zero (i.e., whether there is equivalence) 
does \textit{not}. A deeper investigation of the behaviour of $\{\vec{\theta}^*_n\}_{n\in\mathbb{N}}$ is interesting, 
but is beyond the scope of this paper. 

We first extend Theorem~\ref{th:Limit} for the case when the tuning parameter $\vec{\theta}^*$ 
depends on $n$. 

\begin{lemma} 
Consider the microcanonical ensemble defined in \eqref{eq:PM} with constraint $\vec{T} = \vec{T}_n^*$ 
defined as in \eqref{operator}, and the canonical ensemble defined in \eqref{eq:CPD}--\eqref{eq:PF} 
with parameter $\vec{\theta} = \vec{\theta}^*_n$ such that \eqref{softconstraint2} holds. If the conditions 
in Remark~{\rm\ref{rem:LMn}} hold, then \eqref{varreprsinfty} holds too.  
\end{lemma}

\begin{proof}
The proof of Theorem~\ref{th:Limit} carries over to the setting in which the parameter $\vec{\theta}^*$ 
depends on $n$, i.e., $\vec{\theta}^*= \vec{\theta}^*_n$. The only non-trivial step is to show that 
\begin{equation} 
\label{thetainf}
\lim_{n\to\infty} \psi_{n}(\vec{\theta}^*_n) = \psi_{\infty}(\vec{\theta}^*_{\infty}).
\end{equation}
In the proof of Theorem~\ref{th:Limit} we have shown the pointwise convergence 
\begin{equation}
\lim_{n\to\infty} \psi_{n}(\vec{\theta}) = \psi_{\infty}(\vec{\theta}),
\end{equation}
for every $\vec{\theta}\in\mathbb{R}^m$, independently of $n$. A straightforward computation shows 
that $\nabla\psi_n(\vec{\theta}) = (\langle T_1\rangle,\hdots,\langle T_m\rangle)$, recall 
\eqref{softconstraint2} . Observe that for the specific choice of the parameter $\vec{\theta} 
= \vec{\theta}^*_n=\vec{\theta}^*$, we have that  $\nabla\psi_n(\vec{\theta}^*_n) = ( T_1^*,\hdots, T_m^*)$, 
which yields $\|\nabla\psi_n(\vec{\theta})\|\leq m$ for all $n\in\N$ and $\vec{\theta}\in\mathbb{R}^m$. 
We prove \eqref{thetainf} under the assumptions made in Remark~\ref{rem:LMn},
\begin{align}
|\psi_n(\vec{\theta}^*_n) - \psi_{\infty}(\vec{\theta}^*_\infty)|
&\leq |\psi_n(\vec{\theta}^*_n )- \psi_n(\vec{\theta}^*_{\infty})|
+|\psi_n(\vec{\theta}^*_{\infty}) - \psi_{\infty}(\vec{\theta}^*_\infty)|\\ \nonumber
&\leq \|\nabla\psi_n(\vec{{\eta}})\|\, \|\vec{\theta}^*_n-\vec{\theta}^*_\infty\| 
+ |\psi_n(\vec{\theta}^*_{\infty}) - \psi_{\infty}(\vec{\theta}^*_\infty)|\\ \nonumber
&\leq m\,\|\vec{\theta}^*_n-\vec{\theta}^*_\infty\| + |\psi_n(\vec{\theta}^*_{\infty}) 
- \psi_{\infty}(\vec{\theta}^*_\infty)| \to 0, \quad n \to \infty,
\end{align}
where the second inequality follows from the mean-value theorem for some $\vec{\eta}=c\,\vec{\theta}_n^* 
+ (1-c)\,\vec{\theta}_{\infty}^*$, $c\in(0,1)$. The rest of the proof of Theorem~\ref{th:Limit} carries over intact. 
\end{proof}

In the following lemma we extend the result of Lemma \ref{NMAX} for the case the operator $\vec{T}$ 
is the triangle density $T_2$ . This extension is needed in the proof of Theorem \ref{thm:equivalence} (I). 

\begin{lemma}
\label{lem:NMAX}
Consider the operator $T_2\colon\,\tilde{W} \to \R$ which is bounded and continuous with respect to the 
$\delta_{\square}$-norm as defined in \eqref{deltam}. For $n\in\N$, consider the tuning parameter 
$\theta^*_n$ according to \eqref{softconstraint2}, i.e., 
\begin{equation} 
\sum_{G\in\cG_n}T_2(G)\,\Pcan(G) = T_2^*.
\end{equation}
Suppose that $T_2^*\geq\tfrac{1}{8}$ and that the limits $T_\infty^*, \theta^*_\infty$ in \eqref{eq:Assumption} 
exists. Then 
\begin{equation}
\lim_{n\to\infty} \sum_{G\in\cG_n}T_2(G)\,\Pcan(G)
=\lim_{n\to\infty} \frac{\sum_{G\in\cG_n} T_2(G)\,\eee^{n^2\theta_n^*T_2(G)}}
{\sum_{G\in\cG_n} \eee^{n^2\theta_n^*T_2(G)}} = u^*(\theta^*_{\infty}),
\end{equation}
where
\begin{equation}
\label{eq: A9}
u^*(\theta) = \text{arg}\sup_{0\leq u\leq 1} [\theta u^3- I(u)].
\end{equation}
\end{lemma}

\begin{proof}
From Lemma~\ref{lem: fixed n sign}, since $T_2^*\geq\tfrac{1}{8}$ we have that $\theta^*_n \geq 0$ 
for all $n$. Consequently, $\theta^*\geq 0$. Define, for $\theta\geq0$, the function 
\begin{equation}
f_n(\theta):=\sum_{G\in\cG_n} T_2(G)\Pcan(G\mid\vec{\theta}) 
= \frac{\sum_{G\in\cG_n}T_2(G)\,\eee^{n^2\theta T_2(G)}}
{\sum_{G\in\cG_n} \eee^{n^2\theta T_2(G)}}
\end{equation}
and consider the variational problem in \eqref{eq:variational}. From Chatterjee and Diaconis~\cite{CD13} 
we have that, for $\theta\geq 0$,
\begin{equation}
\psi_{\infty}(\theta):=\sup_{\tilde{h}\in\tilde{W}}\big[\theta T(\tilde{h}) - I(\tilde{h})\big] 
= \sup_{0\leq u\leq1}\big[\theta u^3 - I(u)\big].
\end{equation}
From Radin and Sadun~\cite[Theorem 2.1]{RS15} we have that the function $\theta\to
u^*(\theta)$ is differentiable on $[0,\infty)$. We also observe that
\begin{equation}
\label{eq: properties}
u^*(0) = \tfrac{1}{2}, \qquad \lim_{\theta\rightarrow\infty}u^*(\theta) = 1.
\end{equation}
Moreover, for very $n$, $\theta \mapsto f_n(\theta)$ is continuous on $[0,\infty)$. Hence, combining 
Lemma~\ref{NMAX}, the continuity of $f_n$ for every $n$, the analyticity of the limiting 
function $\theta\mapsto u^*(\theta)$ and (\ref{eq: properties}), we obtain that if the limit 
$\theta_{\infty}$ in \eqref{eq:Assumption} exists, then
\begin{equation}
\lim_{n\to\infty} f_n(\theta^*_n) = u^*(\theta^*_{\infty}) = T^*_\infty,
\end{equation}
which proves the claim. 
\end{proof}

%%%%%%%%%%%%%%%%%%%%%%%%%%%%%%%%%%%%%%%%%%%%%%%%%%%%%%%%%%%%%%%%%%%
%%                                                               %%
%% Use the two commands below for producing your bibliography    %%
%% with bibtex, then comment again the commands and include the  %%
%% content of the .bbl file in this file below the commands.     %%
%%                                                               %%
%%%%%%%%%%%%%%%%%%%%%%%%%%%%%%%%%%%%%%%%%%%%%%%%%%%%%%%%%%%%%%%%%%%

%\bibliographystyle{amsplain}
%\bibliography{yourbibfilename}

% add below the content of your .bbl file produced by bibtex.

\end{document}